\documentclass[11pt,a4paper]{article}

\usepackage{amsmath,amssymb,amsthm,amsbsy,amsfonts}
\usepackage{pdfrender,xcolor}
\usepackage{tabu}
\usepackage{verbatim}
\usepackage{float}
\usepackage{dsfont}

\PassOptionsToPackage{linktocpage}{hyperref}

\usepackage[colorlinks=true, linkcolor=blue, urlcolor=blue,citecolor=red,pagebackref=true]{hyperref}

\usepackage{tikz}
\usepackage{circuitikz}
\usetikzlibrary{matrix}
\usepackage{tkz-base}
\usetikzlibrary{calc}
\usepackage{tikz-cd}
\usetikzlibrary{arrows}
\usetikzlibrary{calc,fadings,decorations.pathreplacing}
\usetikzlibrary{shapes.geometric}
\usepackage{slashed}
\usepackage{comment}
\usepackage{upgreek}
\usepackage{multirow}
\usepackage{mathrsfs}  
\usepackage{tocloft}
\usepackage{makecell}
\usepackage{longtable}
\usepackage{chngcntr}

\newtheorem{thm}{Theorem}[section]
\newtheorem{cor}[thm]{Corollary}
\newtheorem{lem}[thm]{Lemma}
\newtheorem{prop}[thm]{Proposition}

\newtheorem*{thmm}{Main Theorem}

\theoremstyle{definition}
\newtheorem{rem}[thm]{Remark}

\counterwithin*{case}{subsubsection}
\makeatletter\@addtoreset{case}{thm}\makeatother

\cftsetindents{section}{0em}{2em}
\cftsetindents{subsection}{0em}{2em}

\usepackage{bookmark}

\allowdisplaybreaks

\numberwithin{equation}{section}

\newcommand{\define}[1]{%
	\textit{#1}%
}

\usepackage{orcidlink}

\DeclareMathOperator{\Hom}{Hom}
\DeclareMathOperator{\End}{End}
\DeclareMathOperator{\Aut}{Aut}

\DeclareMathOperator{\Id}{Id}

\DeclareMathOperator{\im}{Im}

\DeclareMathOperator{\Spec}{Spec}

\DeclareMathOperator{\sym}{Sym}

\DeclareMathOperator{\Tr}{Tr}

\DeclareMathOperator{\dvol}{dvol}

\DeclareMathOperator{\Ad}{Ad}
\DeclareMathOperator{\ad}{ad}

\DeclareMathOperator{\ind}{Index}
\DeclareMathOperator{\diag}{diag}
\DeclareMathOperator{\Cas}{Cas}

\DeclareMathOperator{\Ind}{Ind}

\DeclareMathOperator{\Ric}{Ric}
\DeclareMathOperator{\virtualdim}{virtual-dim}

\DeclareMathOperator{\ch}{ch}

\DeclareMathOperator{\Sf}{sf}

\DeclareFontFamily{U}{MnSymbolC}{}
\DeclareSymbolFont{MnSyC}{U}{MnSymbolC}{m}{n}
\DeclareFontShape{U}{MnSymbolC}{m}{n}{
	<-6>  MnSymbolC5
	<6-7>  MnSymbolC6
	<7-8>  MnSymbolC7
	<8-9>  MnSymbolC8
	<9-10> MnSymbolC9
	<10-12> MnSymbolC10
	<12->   MnSymbolC12}{}
\DeclareMathSymbol{\intprod}{\mathbin}{MnSyC}{'270}

\begin{document}
\begin{titlepage}
	\vskip 2.0cm

\begin{center}

	{\Large\bf
		Deformations of Clarke--Oliveira's Instantons on\\ Bryant--Salamon $Spin(7)$-Manifold
	}
			
		\vspace{8mm}
			
		{\large Tathagata~Ghosh \orcidlink{0000-0001-6864-0854}
			}
			\\[4mm]
			\noindent {\em School of Mathematics, University of Leeds,\\ Woodhouse Lane, Leeds, LS2 9JT, UK}\\
			{Email: \href{mailto:T.Ghosh@leeds.ac.uk}{T.Ghosh@leeds.ac.uk}}			
			\vspace{8mm}

\begin{abstract}
	In this paper we compute the deformations of Clarke--Oliveira's instantons on the Bryant--Salamon $Spin(7)$-Manifold. The Bryant--Salamon $Spin(7)$-Manifold --- the negative spinor bundle of $S^4$ --- is an asymptotically conical manifold where the link is the squashed $7$-sphere. We use the deformation theory developed by the author in a previous paper to calculate the deformations of Clarke--Oliveira's instantons and calculate the virtual dimensions of the moduli spaces.
\end{abstract}

\end{center}

\makeatletter
\renewcommand\@dotsep{10000}   
\makeatother
\tableofcontents

\end{titlepage}


\section{Introduction}\label{section1}
The aim of this paper is to study the deformations of Clarke--Oliveira's instantons on the Bryant--Salamon $Spin(7)$-Manifold and compute the virtual dimensions of the moduli spaces. The Bryant--Salamon $Spin(7)$-Manifold $\slashed{S}^-(S^4)$ is an asymptotically conical $Spin(7)$-manifold. It was the first example of a complete Riemannian manifold whose holonomy is $Spin(7)$. The manifold is equipped with a $Spin(7)$-structure: a $4$-form $\Phi$ on the manifold which is integrable, i.e., $\nabla\Phi = 0$ for the Levi--Civita connection $\nabla$. Now, given a nearly $G_2$-manifold, i.e., a Riemannian $7$-manifold equipped with a $3$-form $\phi$ satisfying $d\phi = 4*\phi$, the cone over it has an integrable $Spin(7)$ structure. Then, an asymptotically conical manifold is a $Spin(7)$-manifold whose integrable $Spin(7)$-structure is asymptotic to integrable $Spin(7)$-structure of the cone over the nearly $G_2$-manifold (usually referred to as the link). The link for the Bryant--Salamon Manifold is the squashed sphere $\Sigma^7 := \frac{Sp(2) \times Sp(1)}{Sp(1) \times Sp(1)}$ (see \cite{bryant1989construction}).\par 
An instanton on a $Spin(7)$-manifold is a connection $A$ on certain bundle on the manifold whose curvature $F_A$ satisfies the $Spin(7)$-instanton equation given by $*(\Phi \wedge F_A) = -F_A$. The first attempt to construct an instanton on the Bryant--Salamon manifold was made by Clarke in \cite{clarke2014instantons}. However, that instanton turned out to have a singularity. Later, Clarke and Oliveira in \cite{clarke-oli2020instantons} constructed non-singular instantons on the Bryant--Salamon $Spin(7)$-Manifold which we refer to as Clarke--Oliveira's instantons. Other than the FNFN instanton on $\mathbb{R}^8$ (see \cite{me2023paper}), Clarke--Oliveira's instantons are the only known instantons on asymptotically conical $Spin(7)$-manifolds.\par 
Clarke--Oliveira's instantons are a one-parameter family of instantons for a positive parameter $y_0 > 0$ on a trivial bundle. There is also a limiting instanton, obtained by taking $y_0 \to \infty$ which is on a non-trivial bundle. In this paper we study the deformations of both the one-parameter family of instantons and the limiting instanton.\par
Clarke--Oliveira's instantons are asymptotically conical instantons, i.e., they also exhibit similar asymptotic behaviour as the manifold --- they converge to an instanton on the squashed sphere at infinity. Explicitly, the Clarke--Oliveira's instantons $A$ converge to the canonical connection $A_\Sigma$ on 
$\Sigma^7$ with the fastest rate of convergence of $\nu = -2$. The moduli space $\mathcal{M}(A_\Sigma, \nu)$ of Clarke--Oliveira's instantons is the space of asymptotically conical instantons on Bryant--Salamon manifold converging to the canonical connection $A_\Sigma$ at infinity with rate $\nu \in (-2,0)$ quotiented by the weighted gauge group consisting of gauge transformations converging to identity at infinity.

The main result on the deformations of Clarke--Oliveira's instantons is given by the following theorem.
\begin{thmm}
	The virtual dimension of the moduli space of Clarke--Oliveira's one parameter family of instantons $A_{y_0}$ with decay rate $\nu \in (-2, 0)$ is given by
	\begin{align}
		\virtualdim \mathcal{M}_{y_0}(A_\Sigma, \nu) = 
			1.
	\end{align}
    The virtual dimension of the moduli space of Clarke--Oliveira's limiting instanton $A_{\text{lim}}$ with decay rate $\nu \in (-2, 0)$ is given by
	\begin{align}
		\virtualdim \mathcal{M}_{\text{lim}}(A_\Sigma, \nu) = 
			-1.
	\end{align}
\end{thmm}
To calculate the deformations of the instantons and calculate the virtual dimensions, we use the deformation theory of asymptotically conical $Spin(7)$-instantons developed in \cite{me2023paper}, very briefly described as follows.\par
Let $X$ be an asymptotically conical manifold with link $\Sigma$ and $A$ be an instanton on a principal $G$-bundle $P \to X$. Consider two Dirac operators: $\slashed{\mathfrak{D}}_A^-$ on $X$ twisted by the instanton $A$, and the Dirac operator $\slashed{\mathfrak{D}}_{A_\Sigma}$ on $\Sigma$ twisted by the connection $A_\Sigma$.

\begin{thm}\label{indjump}\cite{me2023paper}
    The twisted Dirac operator
    $$\slashed{\mathfrak{D}}_A^- : W^{k+1,2}_{\nu-1}(\slashed{S}^-(X) \otimes \mathfrak{g}_P) \to W^{k,2}_{\nu-2}(\slashed{S}^+(X) \otimes \mathfrak{g}_P)$$
    is Fredholm if $\nu+\frac{5}{2} \in \mathbb{R} \setminus\Spec\slashed{\mathfrak{D}}_{A_\Sigma}$. For two weights $\nu, \nu'$ with $\nu \leq \nu'$ for which the corresponding Dirac operators are Fredholm, we have
    $$\ind_{\nu'}\slashed{\mathfrak{D}}_A^- - \ind_{\nu}\slashed{\mathfrak{D}}_A^- = \sum\limits_{\nu < \lambda < \nu'}\dim \ker\left(\slashed{\mathfrak{D}}_{A_\Sigma}-\lambda-\frac{5}{2}\right).$$
\end{thm}
Here, $W^{k,p}_{\nu}(\slashed{S}^\mp(X) \otimes \mathfrak{g}_P)$ is the weighted Sobolev space of sections of the negative (positive) spinor bundle $\slashed{S}^\mp(X)$ twisted by the associated adjoint vector bundle $\mathfrak{g}_P$ of $P$.\par
Let $\mathcal{I}(A, \nu)$ and $\mathcal{O}(A, \nu)$ be the \define{space of infinitesimal deformations} and  \define{obstruction space} respectively, that is, the kernel and cokernel respectively of the twisted Dirac operator $\slashed{\mathfrak{D}}_A^-$ corresponding to the rate $\nu$. Let $\mathscr{D}(\slashed{\mathfrak{D}}_A^-)$ be the set of rates $\nu$ such that $\nu+\frac{5}{2} \in \Spec\slashed{\mathfrak{D}}_{A_\Sigma}$. The main theorem for the deformations of AC $Spin(7)$-instantons in given as follows.
\begin{thm}\cite{me2023paper}\label{mainthm}
	Let $A$ be an AC $Spin(7)$-instanton asymptotic to a nearly $G_2$ instanton $A_\Sigma$. Moreover, Then, for $\nu \in (\mathbb{R} \setminus \mathscr{D}(\slashed{\mathfrak{D}}_A^-)) \cap (-6,0)$, there exists an open neighbourhood $\mathcal{U}(A, \nu)$ of $0$ in $\mathcal{I}(A, \nu)$ and a smooth map $\kappa : \mathcal{U}(A, \nu) \to \mathcal{O}(A, \nu)$, with $\kappa(0) = 0$, such that an open neighbourhood of $0 \in \kappa^{-1}(0)$ is homeomorphic to an open neighbourhood of $A$ in $\mathcal{M}(A_\Sigma, \nu)$. Thus, the virtual dimension of the moduli space is given by the index of the Dirac operator $\slashed{\mathfrak{D}}_A^-$. Moreover, $\mathcal{M}(A_\Sigma, \nu)$ is a smooth manifold if $\mathcal{O}(A, \nu) = \{0\}$.
\end{thm}
The rest of this article is organised as follows. After reconstructing the Bryant--Salamon metric and the Clarke--Oliveira instanton using homogeneous space techniques in sections \ref{section2} and \ref{section3} respectively, we calculate the eigenvalues of the Dirac operator $\slashed{\mathfrak{D}}_{A_\Sigma}$ on the squashed sphere in section \ref{section4}. This enables us to identify the weights for which the Dirac operator is not Fredholm, as well as the jump in the index. However, we still need to compute the indices of the Dirac operators corresponding to a particular weight. This is addressed in section \ref{section5}, where we calculate the indices corresponding to the weight $-5/2$.
 
\section*{Acknowledgment}
I would like to thank my PhD supervisor Derek Harland for his guidance on this project. Special thanks to Johannes Nordstr\"om for his helpful comments and discussions.

\section{Bryant--Salamon \texorpdfstring{$Spin(7)$}{Spin(7)}-Manifold}\label{section2}
In this section, we derive the Bryant--Salamon metric using homogeneous space techniques, where we identify the link --- the squashed $7$-sphere --- with the homogeneous space $\frac{Sp(2) \times Sp(1)}{Sp(1) \times Sp(1)}$.
\subsection{The Squashed \texorpdfstring{$7$}{7}-Sphere}
Friedrich--Kath--Moroianu--Semmelmann in \cite{friedrich1997paper} have classified all compact, simply connected homogeneous nearly $G_2$ manifolds. As a homogeneous space, the nearly $G_2$-manifold --- the squashed $7$-sphere can be written as $\Sigma^7 := \frac{Sp(2) \times Sp(1)}{Sp(1) \times Sp(1)}$. 
Denote
$$Sp(1)_u := \left\{\left(\begin{pmatrix}
		g &0\\ 0 &1
	\end{pmatrix}, 1\right) : g \in Sp(1)\right\}, \ \ \ Sp(1)_d := \left\{\left(\begin{pmatrix}
		1 &0\\ 0 &g
	\end{pmatrix}, g\right) : g \in Sp(1)\right\}.$$
 Then,
 $$\mathfrak{sp}(1)_u \oplus \mathfrak{sp}(1)_d = \left\{\left(\begin{pmatrix}
		x &0\\ 0 &y
\end{pmatrix}, y\right) : x, y \in \mathfrak{sp}(1)\right\}.$$
We have a decomposition of the Lie algebra $\mathfrak{sp}(2) \oplus \mathfrak{sp}(1)$ as
$$\mathfrak{sp}(2) \oplus \mathfrak{sp}(1) = \mathfrak{sp}(1)_u \oplus \mathfrak{sp}(1)_d \oplus \mathfrak{m}.$$
We want to find $\mathfrak{m} = (\mathfrak{sp}(1)_u \oplus \mathfrak{sp}(1)_d)^\perp$, where the orthogonality is with respect to the Killing form, and choose a basis. We note that $\mathfrak{m} \cong T_p\Sigma \cong V_p \oplus H_p \cong \im \mathbb{H} \oplus \mathbb{H}$, where $V_p$ is the vertical space and $H_p$ is the horizontal space with dimensions $3$ and $4$ respectively, corresponding to the Hopf fibration $S^7 \to S^4$. Then, using the Killing form, we have,
$$\im \mathbb{H} \cong \left\{\left(\begin{pmatrix}
	0 &0\\ 0 &2z
\end{pmatrix}, -3z\right) : z \in \mathfrak{sp}(1)\right\}$$
and, we choose a basis
\begin{align}\label{Ibasis}
    I_1 = \left(\begin{pmatrix}
	0 &0\\ 0 &2i
\end{pmatrix}, -3i\right),\ I_2 = \left(\begin{pmatrix}
	0 &0\\ 0 &2j
\end{pmatrix}, -3j\right),\ I_3 = \left(\begin{pmatrix}
	0 &0\\ 0 &2k
\end{pmatrix}, -3k\right).
\end{align}
Moreover,
$$\mathbb{H} = \left\{\left(\begin{pmatrix}
	0 &b\\ -b^\dagger &0
\end{pmatrix}, 0\right) : b \in \mathbb{H}\right\}$$
and, we choose a basis
\begin{align}
    I_4 = \left(\begin{pmatrix}
	0 &1\\ -1 &0
\end{pmatrix}, 0\right),\ I_5 = \left(\begin{pmatrix}
	0 &-i\\ -i &0
\end{pmatrix}, 0\right),\ I_6 = \left(\begin{pmatrix}
	0 &-j\\ -j &0
\end{pmatrix}, 0\right),\ I_7 = \left(\begin{pmatrix}
	0 &-k\\ -k &0
\end{pmatrix}, 0\right).
\end{align}
Denote the dual basis of $I_a$ by $e^a$ for $a = 1, \dots, 7$.\par 
Then $I_1, \dots, I_7$ together with
\begin{align}
    I_8 = \left(\begin{pmatrix}
	i & 0\\ 0 &0
\end{pmatrix}, 0\right),\ I_9 = \left(\begin{pmatrix}
	j &0\\ 0 &0
\end{pmatrix}, 0\right),\ I_{10} = \left(\begin{pmatrix}
	k &0\\ 0 &0
\end{pmatrix}, 0\right),\nonumber\\
I_{11} = \left(\begin{pmatrix}
	0 & 0\\ 0 &i
\end{pmatrix}, i\right),\ I_{12} = \left(\begin{pmatrix}
	0 &0\\ 0 &j
\end{pmatrix}, j\right),\ I_{13} = \left(\begin{pmatrix}
	0 &0\\ 0 &k
\end{pmatrix}, k\right)
\end{align}
form a basis of $\mathfrak{sp}(2) \times \mathfrak{sp}(1)$. Our objective is to calculate the $Sp(2) \times Sp(1)$-invariant metric $g$, three-form $\phi$ and $\psi = *\phi$ on $\Sigma$. We note that this corresponds to calculating the $Sp(1)_u \times Sp(1)_d$-invariant metric $g$, three-form $\phi$ and $\psi = *\phi$ on $\mathfrak{m}$. We consider an ansatz for $\phi$ given by
\begin{align}\label{sqphi}
    \phi = \alpha^3 e^{123} - \alpha\beta^2(e^1\wedge \omega_1 + e^2\wedge \omega_2 + e^3\wedge \omega_3)
\end{align}
where $\omega_1, \omega_2, \omega_3$ forms a basis for $\Lambda^2_+\mathbb{H}^*$. Explicitly, we take $\omega_1 = e^{45}+e^{67}, \omega_2 = e^{46}-e^{57}, \omega_3 = e^{47}+e^{56}$. Then, we can write $\psi = *\phi$, the metric $g$ and the volume form as
$$\psi = \frac{1}{6}\beta^4(\omega_1 \wedge \omega_1 + \omega_2 \wedge \omega_2 + \omega_3 \wedge \omega_3) - \alpha^2\beta^2(e^{12}\wedge\omega_3 + e^{23}\wedge\omega_1 + e^{31}\wedge\omega_2)$$
$$g = \alpha^2 \sum\limits_{i = 1}^3 e^i \otimes e^i + \beta^2 \sum\limits_{j = 4}^7 e^j \otimes e^j$$
and $\dvol = \alpha^3\beta^4e^{1234567}$ respectively. Hence,
\begin{align*}
    \phi &= \alpha^3 e^{123} - \alpha\beta^2(e^{145} + e^{167} + e^{246} - e^{257} + e^{347} + e^{356}),\\
    \psi &= \beta^4e^{4567} - \alpha^2\beta^2(e^{1247}+e^{1256} + e^{2345}+e^{2367} - e^{1346} + e^{1357}).
\end{align*}
It is easy to verify that $\phi$ is indeed $Sp(1)_u \times Sp(1)_d$-invariant. We need to find $\alpha$ and $\beta$ such that the metric determined by $\phi$ is the squashed metric, which is a nearly parallel metric. That is, we need to solve $d\phi = 4\psi$.\par
Let $f^C_{AB}$ be the structure constants for the basis $I_A$ dual to $e^A$, i.e., $[I_A, I_B] = f^C_{AB}I_C$. From the Maurer--Cartan equation $de^a = -f^a_{ib}e^i \wedge e^b - \frac{1}{2}f^a_{bc}e^b \wedge e^c$ for $i \in \{8, \dots, 13\}$ and $\ a,b,c \in \{1, \dots, 7\}$, and explicitly calculating the structure constants, $d\phi = 4\psi$ implies $\alpha = 3, \beta = \pm\frac{3}{\sqrt{5}}$.
Hence,
\begin{align}\label{sqmetric}
    g = 9\sum\limits_{i = 1}^3 e^i \otimes e^i + \frac{9}{5} \sum\limits_{j = 4}^7 e^j \otimes e^j.
\end{align}
is the ``squashed" metric on $\Sigma^7$.
An orthonormal basis of $\mathfrak{m}$ is given by
$$\widehat{I}_1 = \frac{1}{3}\left(\begin{pmatrix}
	0 &0\\ 0 &2i
\end{pmatrix}, -3i\right),\ \widehat{I}_2 = \frac{1}{3}\left(\begin{pmatrix}
	0 &0\\ 0 &2j
\end{pmatrix}, -3j\right),\ \widehat{I}_3 = \frac{1}{3}\left(\begin{pmatrix}
	0 &0\\ 0 &2k
\end{pmatrix}, -3k\right).$$
$$\widehat{I}_4 = \frac{\sqrt{5}}{3}\left(\begin{pmatrix}
	0 &1\\ -1 &0
\end{pmatrix}, 0\right),\ \widehat{I}_5 = \frac{\sqrt{5}}{3}\left(\begin{pmatrix}
	0 &-i\\ -i &0
\end{pmatrix}, 0\right),$$
$$\widehat{I}_6 = \frac{\sqrt{5}}{3}\left(\begin{pmatrix}
	0 &-j\\ -j &0
\end{pmatrix}, 0\right),\ \widehat{I}_7 = \frac{\sqrt{5}}{3}\left(\begin{pmatrix}
	0 &-k\\ -k &0
\end{pmatrix}, 0\right).$$
We denote the dual basis by $\widehat{e}^i$ for $i = 1, \dots, 7$.
\subsection{The Bryant--Salamon Metric}
We just studied the squashed sphere 
$\Sigma^7 := \frac{Sp(2) \times Sp(1)}{Sp(1) \times Sp(1)}$ as a nearly $G_2$ manifold where the $G_2$-structure is given by (\ref{sqphi}).\par
Now, we consider $(0, \infty) \times \Sigma^7$ equipped with the $Spin(7)$-structure $\Phi = dr \wedge \phi + \psi$ where we consider $\alpha, \beta$ (and hence $\phi, \psi$) as functions of $r$. The metric is given by 
\begin{align}\label{conemet}
    g = dr^2 + \alpha(r)^2 \sum\limits_{i = 1}^3 e^i \otimes e^i + \beta(r)^2 \sum\limits_{j = 4}^7 e^j \otimes e^j.
\end{align}
The metric has holonomy $Spin(7)$ iff $\Phi$ is closed. Then, we have, $\displaystyle\frac{\partial\psi}{\partial r} = d_\Sigma\phi$, which implies,
which implies,
\begin{align}
    &\ \ \frac{d \beta^2}{d r} = \frac{6}{5}\alpha\label{parb3r}\\
    &\Rightarrow \frac{d \beta}{d r} = \frac{3\alpha}{5\beta}.\label{parbeta}
\end{align}
and 
\begin{align}\label{paralpha}
    \frac{d \alpha}{d r} = \frac{25\beta^2 - 2\alpha^2}{5\beta^2}
\end{align}
Hence,
\begin{align*}
    \frac{d \beta}{d \alpha} = \frac{3\alpha\beta}{25\beta^2 - 2\alpha^2}.
\end{align*}
This is a homogeneous ordinary differential equation. The solution is $\beta^4(5\beta^2-\alpha^2)^3 = C$.\par 
Now, with the initial condition $\alpha(0) = 0$ and $\beta(0) =: \beta_0$, we have $\beta^4(5\beta^2-\alpha^2)^3 = \beta_0^{10}$, and $\alpha, \beta$ are both strictly increasing for $r >0$. It can be shown that the metric (\ref{conemet}) on $(0, \infty) \times \Sigma^7$ can be smoothly extended over $((0, \infty) \times \Sigma^7) \cup S^4$.\par 
Now,
\begin{align}\label{alphabeta}
    &\ \ \beta^4(5\beta^2-\alpha^2)^3 = \beta_0^{10} \Rightarrow \alpha^2 = \left(5 - (\beta_0\beta^{-1})^{\frac{10}{3}}\right)\beta^2.
\end{align}
Moreover,
\begin{align*}
    dr^2 &= \left(\frac{\partial r}{\partial \beta}\right)^2d\beta^2 = \frac{25}{9}\frac{1}{5 - (\beta_0\beta^{-1})^{\frac{10}{3}}}d\beta^2.
\end{align*}
Moreover, from (\ref{parb3r}), We note that
\begin{align}\label{intalpha}
    \beta^2(r) = \beta_0^2 + \frac{6}{5}\int_0^r\alpha(s)\ ds.
\end{align}
Hence, considering $\beta$ as an independent variable, the metric (\ref{conemet}) can be written as
\begin{align}\label{bsmetric}
    g = \frac{25}{9}\frac{1}{5 - (\beta_0\beta^{-1})^{\frac{10}{3}}}d\beta^2 + \left(5 - (\beta_0\beta^{-1})^{\frac{10}{3}}\right)\beta^2\sum\limits_{i = 1}^3 e^i \otimes e^i + \beta^2 \sum\limits_{j = 4}^7 e^j \otimes e^j
\end{align}
which is the \define{Bryant--Salamon metric} on $((0, \infty) \times \Sigma^7) \cup S^4 \cong \slashed{S}^-(S^4)$. Thus, $\slashed{S}^-(S^4)$ is an asymptotically conical $Spin(7)$-manifold over the link squashed sphere with rate $-10/3$.

\section{Clarke-Oliveira's Instantons}\label{section3}
In this section we reconstruct Clarke--Oliveira's instantons \cite{clarke-oli2020instantons} using homogeneous space techniques.\par
Let $W_i$ be the unique irreducible representation of $SU(2) \cong Sp(1)$ of dimension $(i+1)$. Then,\\
$W_0 \equiv$ Trivial representation ($\dim W_0 = 1$),\\
$W_1 \equiv$ Standard representation ($\dim W_1 = 2$),\\
$W_2 \equiv$ Adjoint representation ($\dim W_2 = 3$).\par
Let $W_i^u$ be an irreducible representation of $Sp(1)_u$ and $W_i^d$ be an irreducible representation of $Sp(1)_d$. Then $W_{(i, j)} := W_i^u \otimes W_j^d$ is an irreducible representation of $Sp(1)_u \times Sp(1)_d$. Clearly, $\dim W_{(i,j)} = (i+1)(j+1)$. Then 
\begin{align}\label{m-as-irrep}
    \mathfrak{m} = W_{(1,1)} \oplus W_{(0,2)}.
\end{align}
is the decomposition of $\mathfrak{m}$ into irreducible representations of $Sp(1)_u \times Sp(1)_d$.\par
Consider the gauge group $Sp(1) \cong SU(2)$. We want to identify all homogeneous $SU(2)$-bundles over the squashed sphere $\Sigma^7 := \frac{Sp(2) \times Sp(1)}{Sp(1) \times Sp(1)}$, and identify their spaces of invariant connections. Then we have three isotropy homomorphisms from $Sp(1)_u \times Sp(1)_d$ to $Sp(1)$, namely
\begin{align*}
    \lambda_0 : Sp(1)_u \times Sp(1)_d &\to Sp(1)\\
    \left(\begin{pmatrix}
	g_1 &0\\ 0 &g_2
\end{pmatrix}, g_2\right) &\mapsto 1\\
\lambda_1 : Sp(1)_u \times Sp(1)_d &\to Sp(1)\\
    \left(\begin{pmatrix}
	g_1 &0\\ 0 &g_2
\end{pmatrix}, g_2\right) &\mapsto g_1\\
\lambda_2 : Sp(1)_u \times Sp(1)_d &\to Sp(1)\\
    \left(\begin{pmatrix}
	g_1 &0\\ 0 &g_2
\end{pmatrix}, g_2\right) &\mapsto g_2.
\end{align*}
Then, the bundles $P_i = (Sp(2) \times Sp(1)) \times_{\lambda_i} Sp(1)$ for $i = 1,2,3$ are the only homogeneous $SU(2)$-bundles over $\Sigma^7$. From Wang's theorem \cite{wang1958}, it follows that the invariant connections on $P_i$ correspond to the $Sp(1) \times Sp(1)$-equivariant homomorphisms
$$\Lambda_i : (\mathfrak{m}, \ad) \to (\mathfrak{sp}(1), \ad \circ \lambda_i).$$
Now,
$$\ad \circ \lambda_i : \mathfrak{sp}(1)_u \oplus \mathfrak{sp}(1)_d \to \End(\mathfrak{sp}(1)),$$
Then,
$$\ad \circ \lambda_0(X, Y)Z = \ad(0)Z = 0.$$
Hence, the map $\Lambda_0$ is equivalent to a map
$$W_{(1,1)} \oplus W_{(0,2)} \to W_{(0,0)}$$
and so, must be trivial. Moreover,
$$\ad \circ \lambda_1(X, Y)Z = \ad(X)Z = [X,Z].$$
Hence, the map $\Lambda_1$ equivalent to a map
$$W_{(1,1)} \oplus W_{(0,2)} \to W_{(2,0)}$$
is again trivial. Finally,
$$\ad \circ \lambda_2(X, Y)Z = \ad(Y)Z = [Y,Z].$$
Hence, the map $\Lambda_2$ can be described as follows,
$$W_{(1,1)} \oplus W_{(0,2)} \to W_{(0,2)}.$$
That is, by Schur's lemma, $\Lambda_2|_{W_{(1,1)}}$ is trivial map, whereas $\Lambda_2|_{W_{(0,2)}}$ is the map
$$\varphi\cdot\Id : W_{(0,2)} \to W_{(0,2)}$$
for some real number $\varphi$.\par 
We note that for $\lambda_0$ we get the flat connection, and for $\lambda_1$, the canonical connection. Thus, these two cases fail to give us anything interesting. Hence, we ignore these two cases. We rename $\lambda_2$ to be $\lambda$, $\Lambda_2$ to be $\Lambda$ and the corresponding bundle $P_2$ to be $P$.\par 
We note that a basis $I_A$  for $\mathfrak{sp}(2)\oplus\mathfrak{sp}(1)$ can be represented by left invariant vector fields $\widehat{E}_A$ on $Sp(2) \times Sp(1)$ as well as by the dual basis $\hat{e}^A$ of left invariant $1$-forms. Denote the natural projection map
\begin{align*}
	\pi : Sp(2) \times Sp(1) &\to \Sigma^7\\
	g &\mapsto g(Sp(1)_u \times Sp(1)_d),
\end{align*}
of the principal bundle. Let $U$ be a contractible open subset of $\Sigma^7$. Then consider a local section $\sigma$ of the bundle $Sp(2) \times Sp(1) \to \Sigma^7$, i.e., a map $\sigma : U \to Sp(2) \times Sp(1)$ such that $\pi \circ \sigma = \Id_U$. We put $e^A := \sigma^*\hat{e}^A$. Then $\{e^a : a = 1, \dots, \dim \mathfrak{m}\}$ form an orthonormal frame for $T^*(\Sigma^7)$ over $U$.\par 
Let us fix a basis $T_a,\ a = 1,2,3$ for $\mathfrak{sp}(1) \cong \mathfrak{su}(2)$, where $T_a = -i\sigma_a$ and $\sigma_a,\ a = 1,2,3$ are the Pauli matrices given by
$$\sigma_1 = \begin{pmatrix}
	0&1\\ 1&0
\end{pmatrix}\ \ \sigma_2 = \begin{pmatrix}
0&-i\\ i&0
\end{pmatrix}\ \ \sigma_3 = \begin{pmatrix}
1&0\\ 0&-1
\end{pmatrix}.$$ 
Let us denote $\Psi$ by the matrix of $\Lambda$, i.e., $\Lambda(I_a) = \Psi_{ac}T_c$, for $a, c = 1,2,3$.\par
In local coordinates, any $Sp(2) \times Sp(1)$-invariant connection on the bundle $P$ over the nearly $G_2$-manifold $\Sigma^7$ can be written as
\begin{align}\label{homconn}
    A = e^i\lambda(I_i) + e^a\Lambda(I_a)
\end{align}
where $i = 11, 12, 13$, and the basis elements $I_A$ have been listed in (\ref{Ibasis}). $A_\Sigma := e^i\lambda(I_i)$ is the canonical connection. We note that here, the two notions of canonical connection -- one for the homogeneous space and other for the nearly $G_2$-manifold coincide here \cite{singhal2021paper}.\par 
Now, we consider $(0, \infty) \times \Sigma^7$ equipped with the $Spin(7)$-structure $\Phi = dr \wedge \phi + \psi$ and a metric given by (\ref{conemet}). A connection 1-form on $(0,\infty) \times \Sigma^7$ is given by $A = A_0e^0+A_ae^a$ for $e^0 = dr$, which gives the $Sp(2) \times Sp(1)$-invariant connection (\ref{homconn}), but now, we consider $\Lambda$ to be a function of $r$. Now, for $a = 1, \dots, 7$ and $b = 1,2,3$, we have $\Lambda(I_a) = \Psi_{ab}T_b$ where $\Psi_{ab}(r) = \varphi(r)\delta_{ab}$. Then,
\begin{align}\label{COinst}
    A = e^{a+10}T_a + \varphi(r)e^aT_a,
\end{align}
for $a = 1,2,3$. Here without loss of generality, we take the temporal gauge $A_0 = 0$. The curvature of this connection is given by
$$F_A = F_{0a}e^0 \wedge e^a+\frac{1}{2}F_{bc}e^b \wedge e^c$$
where  
$$F_{0a} = \frac{\partial A_a}{\partial r} = \dot{\varphi}(r)T_a.$$
Now, the ASD instanton equation can be written as
$$\frac{1}{\alpha(r)}F_{0a} = -\frac{1}{2}\left[\frac{1}{\alpha(r)^2}\sum\limits_{b,c = 1}^3\phi_{abc}F_{bc} + \frac{1}{\beta(r)^2}\sum\limits_{b,c = 4}^7\phi_{abc}F_{bc}\right]$$
where $\phi_{abc}$ are structure constants of the octonions. Applying the Maurer--Cartan equations,
\begin{align*}
	de^a &= -f^a_{ib}e^i \wedge e^b - \frac{1}{2}f^a_{bc}e^b \wedge e^c\\
	de^i &= -\frac{1}{2}f^i_{bc}e^b \wedge e^c - \frac{1}{2}f^i_{jk}e^j \wedge e^k
\end{align*}
(for $i,j,k \in \{8, \dots, 13\},\ a,b,c \in\{1, \dots, 7\}$) we have
$$(dA)_{bc} = -f^{d+10}_{bc}T_d -\varphi(r)f^d_{bc}T_d$$
and
\begin{align*}
    [A \wedge A]_{bc} &= 4\varphi(r)^2\epsilon_{dbc}T_d.
\end{align*}
Hence,
$$F_{bc} = -f^{d+10}_{bc}T_d - \varphi(r)f^d_{bc}T_d + 2\varphi(r)^2\epsilon_{dbc}T_d.$$
Hence the ASD $Spin(7)$ instanton equation reduces to
\begin{align}
    \frac{2}{\alpha(r)}\dot{\varphi}(r)T_a = \left[\frac{1}{\alpha(r)^2}\sum\limits_{b,c = 1}^3\left(\phi_{abc}f^{d+10}_{bc}T_d + \varphi(r)\phi_{abc}f^d_{bc}T_d - 2\varphi(r)^2\phi_{abc}\epsilon_{dbc}T_d\right)\right.\nonumber\\
    \left.+ \frac{1}{\beta(r)^2}\sum\limits_{b,c = 4}^7\left(\phi_{abc}f^{d+10}_{bc}T_d + \varphi(r)\phi_{abc}f^d_{bc}T_d\right)\right]
\end{align}
From the values of the structure constants, simplifying, we get
\begin{align}\label{instdiff1}
    \alpha\dot{\varphi}(r)  = \left(12-\frac{12\alpha^2}{5\beta^2}\right) - \left(2+\frac{4\alpha^2}{5\beta^2}\right)\varphi(r) - 2\varphi(r)^2.
\end{align}
To solve the equation (\ref{instdiff1}), we first simplify by the substitution $x := \varphi + 3$, which gives,
\begin{align}
    \dot{x} = -\frac{2}{\alpha}x\left(x - \left(5-\frac{2\alpha^2}{5\beta^2}\right)\right).\label{orginseq}
\end{align}
Now, following the analysis done in \cite{clarke-oli2020instantons} we use the substitution
\begin{align}
    &\ \ y(r) = \frac{x(r)}{\alpha(r)^2}\nonumber\\
    &\Rightarrow x = \alpha^2 y\label{1stsubinssol}\\
    &\Rightarrow \dot{x} = \alpha^2\dot{y} + 2\alpha y \frac{25\beta^2 - 2\alpha^2}{5\beta^2}\label{2ndsubinssol}
\end{align}
where we have used (\ref{paralpha}). Substituting (\ref{1stsubinssol}) and (\ref{2ndsubinssol}) in (\ref{orginseq}), we have 

\begin{align}\label{ode1}
    &\ \ \dot{y} = -2\alpha y^2 \Rightarrow \frac{dy}{y^2} = -2\alpha(r) dr.
\end{align}
Now, we consider the initial condition,
\begin{align}\label{inicon}
    y(0) = y_0.
\end{align}
Then, integrating (\ref{ode1}) with the initial condition (\ref{inicon}), we have
\begin{align*}
    y(r) = \frac{1}{\frac{1}{y_0} + 2\int_0^r\alpha(r) dr} \Rightarrow x(r) = \frac{y_0\alpha^2}{1 + 2y_0\int_0^r\alpha(s) ds}.
\end{align*}
Then, from (\ref{alphabeta}) and (\ref{intalpha}), we have
\begin{align}\label{valueofphi}
    \varphi(r) = \frac{\alpha^2}{\frac{1}{y_0} + 2\int_0^r\alpha(s) ds}-3 = \frac{ 5\beta_0^2\beta^{\frac{4}{3}}-\beta_0^{\frac{10}{3}}-\frac{3}{y_0}\beta^{\frac{4}{3}}}{\beta^{\frac{4}{3}}\left(\frac{1}{y_0} + \frac{5}{3}(\beta^2-\beta_0^2)\right)}.
\end{align}
\begin{rem}
    We note that $\varphi(r) = 0$ corresponds to the canonical connection. We want to find the value of $\varphi(r)$ for which we have the flat connection. For flat connection, we can write $\varphi(r) = c$ where $c$ is a constant. To find $c$ we substitute $\varphi = c$ and $\frac{\alpha}{\beta} = \sqrt{5}$ in the above equation. Then, equation (\ref{instdiff1}) implies $c = 0, -3$. It is easy to verify that $c=-3$ corresponds to the flat connection, which takes the form
\begin{align}\label{flatcon}
    A_0 &= e^{a+10}T_a -3e^aT_a\nonumber\\
    &= \widehat{e}^{a+10}T_a -\widehat{e}^aT_a.
\end{align}
for $a = 1,2,3$, where $\widetilde{e}^a$ for $a = 1, \dots, 7$ is an orthonormal basis for the metric (\ref{conemet}), and $e^{a+10} = \widehat{e}^{a+10}$. Now, the underlying manifold being simply connected, the flat connection $A_0$ is the trivial connection (up to gauge). 
\end{rem} 
Thus, we have a real $1$-parameter family of $Spin(7)$-instantons which, following \cite{clarke-oli2020instantons}, we denote by $A_{y_0}$. For $y_0 = 0$, the connection $A_{y_0 = 0}$ is a flat connection, whereas for $y_0 > 0$, $A_{y_0}$ is irreducible. For $y_0 < 0$, the $Spin(7)$-instantons are only locally defined in a neighbourhood of $S^4$ \cite{clarke-oli2020instantons}. As $y_0 \to \infty$, the instanton $A_{y_0}$ and all its derivatives converge to an instanton $A_{\text{lim}}$ away from the zero section $S^4$ \cite{clarke-oli2020instantons}.\par
The following proposition follows from the removable singularity theorem of Tao and Tian \cite{taotian2004}.
\begin{prop}[\cite{clarke-oli2020instantons}]
    The instanton $A_{y_0}$ on $\Sigma^7 \times \mathbb{R} \cong \slashed{S}^-(S^4) \setminus S^4$ smoothly extends over the zero section $S^4$ (up to gauge) if and only if the curvature $F_{A_{y_0}}$ is bounded.
\end{prop}
Then, we have the following theorem.
\begin{thm}[\cite{clarke-oli2020instantons}]
    $\{A_{y_0}\}_{y_0 \in [0,\infty)}$ is a real $1$-parameter family of $Spin(7)$-instantons on the trivial bundle $\slashed{S}^-(S^4) \times \mathbb{C}^2 \to \slashed{S}^-(S^4)$.\par 
    Moreover, $A_{\text{lim}}$ extends smoothly over $S^4$ and gives a $Spin(7)$-instanton on the (non-trivial) bundle $\pi^*(\slashed{S}^-(S^4)) \to \slashed{S}^-(S^4)$, for the projection map $\pi : \slashed{S}^-(S^4) \setminus S^4 \to S^4$.
\end{thm}
Since, for large $r$, we have $\alpha = O(r)$ and $\beta = O(r)$, clearly $\varphi = O(r^{-2})$. 
Then, for the diffeomorphism $h : C(\Sigma) = \Sigma^7 \times \mathbb{R} \to \slashed{S}^-(S^4) \setminus S^4$ and projection $p : C(\Sigma^7) \to \Sigma^7$, we have,
\begin{align*}
    |h^*(A_{y_0}) - p^*(A_\Sigma)|_{g_C} = O(r^{-2-1}).
\end{align*}
where $A_\Sigma = e^{a+10}T_a$. Then, the fastest rate of convergence of Clarke--Oliveira's instantons is $-2$, following the convention of definition 2.15 in \cite{me2023paper}.

\section{Eigenvalues of the Dirac Operators on Squashed 7-Sphere}\label{section4}
In this section, using various representation theoretic and homogeneous space techniques developed in \cite{me2023paper},\cite{driscoll2020thesis},\cite{bar1991dirac},\cite{bar92}, we calculate the eigenvalues of two natural Dirac operators on the squashed $7$-sphere. The results will directly be used to find the critical rates of the negative twisted Dirac operators for Clarke--Oliveira's instantons and in the spectral flow analysis for the index of the Dirac operator.
\subsection{Dirac operators on Homogeneous Nearly \texorpdfstring{$G_2$}{G2}-Manifolds}\label{section4.1}
Let $\Sigma = G/H$ be a reductive homogeneous nearly $G_2$-manifold, i.e., equipped with a $G$-invariant $G_2$-structure $\phi$ satisfying $d\phi = 4\psi$. Since $G/H$ is a reductive homogeneous space, there is an orthogonal decomposition $\mathfrak{g} = \mathfrak{h}\oplus\mathfrak{m}$ induced by the Killing form $K$ on $G$, which is defined by
\begin{equation}\label{killingform}
    K(X, Y) = \Tr_{\mathfrak{g}}(\ad(X)\ad(Y)).
\end{equation}
Let us assume that for some constant $c$ the metric given by
\begin{equation}\label{killmet}
    g(X, Y) = -c^2K(X, Y)
\end{equation}
is a nearly $G_2$-metric. Let $\{I_A\}$ be an orthonormal basis for $\mathfrak{g}$, where $\{I_a : a = 1, \dots, \dim(G/H) = 7\}$ is a orthonormal basis for $\mathfrak{m}$ and $\{I_i : i = \dim(H)+1, \dots, \dim(G)\}$ is a orthonormal basis of $\mathfrak{h}$.\par  
We consider the principal $H$-bundle $G \to \Sigma$, a representation $\rho_V : H \to \Aut(V)$ of $H$, and the associated vector bundle $E := G \times_\rho V$.\par 
Let $\widehat{G}$ be the set of equivalence classes of irreducible representations of $G$. Then for each $\gamma \in \widehat{G}$, we have a representation $(V_\gamma, \rho_\gamma)$ of $G$. Consider the complex spinor bundle $\slashed{S}(\Sigma) = G \times_{\rho, H}\Delta$ where $\Delta$ is the spinor space (that is, an $8$-dimensional representation of $Cl(7)$). From the splitting $\slashed{S}_\mathbb{C}(\Sigma) \cong \Lambda^0_\mathbb{C}\oplus\Lambda^1_\mathbb{C}$ (see \cite{me2023paper}), we have $\Delta \cong \mathbb{C}\oplus\mathfrak{m}_\mathbb{C}^*$. We twist the spinor bundle by the associated vector bundle $E = G \times_{\rho_V, H}V$ for a representation $V$ of $H$. Then 
$$\slashed{S}_\mathbb{C}(\Sigma)\otimes E = G\times_{(\rho_\Delta \otimes \rho_V, H)}\Delta \otimes V.$$ 
Then, from \define{Frobenius reciprocity}, we have the decomposition 
\begin{equation}\label{frob2}
    L^2(\slashed{S}_\mathbb{C}(\Sigma)\otimes E) \cong L^2(G, \Delta \otimes V)^H \cong \bigoplus\limits_{\gamma\in\widehat{G}}\Hom(V_\gamma, \Delta \otimes V)^H\otimes V_\gamma.
\end{equation}
Consider a family of twisted Dirac operators 
\begin{equation}\label{diracfamily}
    \slashed{\mathfrak{D}}^t_{A_\Sigma} = \slashed{\mathfrak{D}}^1_{A_\Sigma}+\frac{(t-1)}{2}\phi
\end{equation}
where $\slashed{\mathfrak{D}}^1_{A_\Sigma}$ is the Clifford product with the canonical connection $\nabla^{1,A_\Sigma}$ (where $\nabla^1$ is the canonical connection for the spinor bundle and $A_\Sigma$ is the canonical connection of $E$) \cite{me2023paper}. For $t = 0$, we denote $\slashed{\mathfrak{D}}_{A_\Sigma} := \slashed{\mathfrak{D}}^0_{A_\Sigma}$, the twisted Dirac operator for the Levi--Civita connection on the spinor bundle. For every $t \in \mathbb{R}$, the twisted Dirac operator $\slashed{\mathfrak{D}}^t_{A_\Sigma}$, restricted to $\Hom(V_\gamma, \Delta \otimes V)^H\otimes V_\gamma$ is given by
\begin{equation}
    \slashed{\mathfrak{D}}^t_{A_\Sigma}|_{\Hom(V_\gamma, \Delta \otimes V)^H\otimes V_\gamma} = \left(\slashed{\mathfrak{D}}^t_{A_\Sigma}\right)_\gamma \otimes \Id
\end{equation}
where $\left(\slashed{\mathfrak{D}}^t_{A_\Sigma}\right)_\gamma :\Hom(V_\gamma, \Delta \otimes V)^H\to\Hom(V_\gamma, \Delta \otimes V)^H$ is the twisted Dirac operator \cite{driscoll2020paper} given by
\begin{equation}\label{diractgamma}
    \left(\slashed{\mathfrak{D}}^t_{A_\Sigma}\right)_\gamma\eta = -I_a \cdot (\eta \circ \rho_{V_\gamma}(I_a))+\frac{t-1}{2}\phi\cdot\eta.
\end{equation}
Here, we have identified $\mathfrak{m}^*$ with the cotangent space $T^*_e\Sigma$, and thus identified the $G_2$-structure $\phi$ as an element of $\Lambda^3(\mathfrak{m}^*)$. This $\phi$ acts on the spinor space and hence on $\eta \in \Hom(V_\gamma, \Delta \otimes V)^H$.\par 
Similarly, we can define the family of untwisted Dirac operators $\mathcal{D}_\Sigma^t$ acting on the spinor bundle $\slashed{S}(\Sigma^7) = G \times_\rho \Delta$, where for $t = 0$, we denote $\mathcal{D}_\Sigma := \mathcal{D}_\Sigma^0$, the untwisted Dirac operator for the Levi–Civita connection on the spinor bundle. Again, by Frobenius reciprocity, the space of sections of the spinor bundle decomposes as a sum of subspaces $\Hom(V_\gamma, \Delta)^H\otimes V_\gamma$. For every $t \in \mathbb{R}$, the untwisted Dirac operator $\mathcal{D}^t_\Sigma$, restricted to $\Hom(V_\gamma, \Delta)^H\otimes V_\gamma$ is given by
\begin{equation}
    \mathcal{D}^t_\Sigma|_{\Hom(V_\gamma, \Delta)^H\otimes V_\gamma} = \left(\mathcal{D}^t_\Sigma\right)_\gamma \otimes \Id
\end{equation}
where $\left(\mathcal{D}^t_\Sigma\right)_\gamma :\Hom(V_\gamma, \Delta)^H\to\Hom(V_\gamma, \Delta)^H$ is the untwisted Dirac operator defined in a similar way to (\ref{diractgamma}).
\begin{rem}\label{untflat}
    We note that the untwisted Dirac operator $\mathcal{D}_\Sigma$ acting on the bundle $\slashed{S}(\Sigma^7)$ can be identified with the twisted Dirac operator $\slashed{\mathfrak{D}}_{A_0}$ acting the bundle $\slashed{S}(\Sigma^7) \otimes \mathfrak{g}_P$ twisted by flat connection $A_0$ (\ref{flatcon}) on the adjoint bundle $\mathfrak{g}_P$. Hence, the eigenvalues of $\mathcal{D}_\Sigma$ and $\slashed{\mathfrak{D}}_{A_0}$ are the same, a fact that will be used later in calculating the spectral flow of connection for the index calculation.
\end{rem}

\subsection{Eigenvalue Bounds for the Twisted Dirac Operators on Squashed \texorpdfstring{$7$}{7}-Sphere}
Since, for Clarke-Oliveira's instantons, the fastest rate of convergence is $-2$, we consider the family of moduli spaces $\mathcal{M}(A_\Sigma, \nu)$ for $\nu \in (-2,0)$. Recall the $Sp(1)$-bundle $P = (Sp(2) \times Sp(1)) \times_{\lambda} Sp(1)$ over $\Sigma^7$ corresponding to the isotropy homomorphism $\lambda$. Denote the adjoint vector bundle $\mathfrak{g}_P$. Then, by Theorem \ref{indjump}, we are interested in the eigenvalues of the twisted Dirac operator $\slashed{\mathfrak{D}}_{A_\Sigma}$ twisted by the bundle $\mathfrak{g}_P$, in the interval $\left(-2+\frac{5}{2}, 0+\frac{5}{2}\right) =  \left(\frac{1}{2}, \frac{5}{2}\right)$. Thus, by (\ref{frob2}), we need to compute the eigenvalues of the Dirac operator corresponding to only a few irreducible representations $V_\gamma$ of $Sp(2) \times Sp(1)$.\par
To identify these irreducible representations of $Sp(2) \times Sp(1)$, we want bounds for the eigenvalues of the Dirac operator in terms of the irreducible representations. The main tool we use here is the fact that for $t = 1/3$, the square of the Dirac operator $\slashed{\mathfrak{D}}^t_{A_\Sigma}$ can be written in terms of the Casimir operator, which has an explicit dependence of the irreducible representations.\par
Let $V_{(a,b)}$ be the irreducible representations of $Sp(2)$ corresponding to the highest weight vector $(a, b)$. Then,\\
$V_{(0,0)} \cong \mathbb{C}$ is the trivial representation,\\
$V_{(0,1)} \cong \mathbb{H}^2$ is the standard representation,\\
$V_{(1,0)}$ is the $5$-dimensional representation under the isomorphism $Sp(2) \cong Spin(5)$.\par 
Define $V_{(a,b,c)} := V_{(a,b)} \otimes W_c$ to be the irreducible representation of $Sp(2) \times Sp(1)$ and let $W_{(a,b)}$ be that of $Sp(1)_u \times Sp(1)_d$.\par
If $\{I_A\}$ is an orthonormal basis of $\mathfrak{g}$, then the \define{Casimir Operator} $\Cas_\mathfrak{g} \in \sym^2(\mathfrak{g})$ is the inverse of the metric on $\mathfrak{g}$, defined by
\begin{align}\label{casop}
    \Cas_\mathfrak{g} = \sum\limits_{A=1}^{\dim{G}}I_A \otimes I_A.
\end{align}
If $(\rho, V)$ is any representation of $\mathfrak{g}$, then
$$\rho(\Cas_\mathfrak{g}) = \sum\limits_{A=1}^{\dim{G}}\rho(I_A)^2$$
The Casimir eigenvalues of the Casimir operator (\ref{casop}) using the nearly $G_2$-metric (\ref{killmet}) for $c^2 = 3/40$, are given by
\begin{align*}
	\rho_{(a,b,c)}\left(\Cas_{\mathfrak{sp}(2)\oplus \mathfrak{sp}(1)}\right) &= c^{\mathfrak{sp}(2)\oplus \mathfrak{sp}(1)}_{(a,b,c)}\Id,\\
	\rho_{(a,b)}(\Cas_{\mathfrak{sp}(1)_u\oplus \mathfrak{sp}(1)_d})&= c^{\mathfrak{sp}(1)_u\oplus \mathfrak{sp}(1)_d}_{(a,b)}\Id.
\end{align*}
where,
\begin{align*}
        c^{\mathfrak{sp}(2)\oplus \mathfrak{sp}(1)}_{(a,b,c)}&= -\frac{5}{9}(4a^2+2b^2+3c^2+4ab+12a+8b+6c),\\
        c^{\mathfrak{sp}(1)_u\oplus \mathfrak{sp}(1)_d}_{(a,b)}&=-\frac{2}{9}(5a^2+3b^2+10a+6b).
\end{align*}
Now, since $(\mathfrak{sp}(1)_\mathbb{C}, \Ad\circ\lambda) = W_{(0,2)}$, we have
$$\Hom\left(V_\gamma, \Delta\otimes\mathfrak{sp}(1)_\mathbb{C}\right)^{Sp(1)_u\times Sp(1)_d} = \Hom\left(V_\gamma, \Delta\otimes W_{(0,2)}\right)^{Sp(1)_u\times Sp(1)_d}.$$
Then, since $c^{\mathfrak{sp}(1)_u\oplus \mathfrak{sp}(1)_d}_{(0,2)} = -16/3$, from \cite{singhal2021paper} we calculate the eigenvalues of $\left(\slashed{\mathfrak{D}}_{A_\Sigma}^{1/3}\right)_\gamma^2$ to be
$-c^\mathfrak{\mathfrak{sp}(2)\oplus \mathfrak{sp}(1)}_\gamma+\frac{1}{9}$ with multiplicities $\dim\Hom\left(V_\gamma, \Delta\otimes W_{(0,2)}\right)^{Sp(1)_u\oplus Sp(1)_d}$.\par 
Then, from \cite{me2023paper}, we have the following theorem.
\begin{thm}\label{eigbnd2}
	Let $V_\gamma = V_{(a,b,c)}$ be an irreducible representation of $Sp(2)\times Sp(1)$. If 
	$$L_\gamma := L_{(a,b,c)} :=  \sqrt{-c^\mathfrak{\mathfrak{sp}(2)\oplus \mathfrak{sp}(1)}_{(a,b,c)}+\frac{1}{9}}-\frac{7}{6} > 0$$
	then $L_\gamma$ is a lower bound on the absolute values of the eigenvalues of $\left(\slashed{\mathfrak{D}}^0_{A_\Sigma}\right)_\gamma$.
\end{thm}
\begin{cor}\label{replist}
	Consider the irreducible representations of $\mathfrak{sp}(2)\oplus \mathfrak{sp}(1)$ given by
		$$V_{(0,0,0)},\ V_{(1,0,0)},\ V_{(0,0,1)},\ V_{(0,1,0)},\ V_{(1,0,1)},\ V_{(0,1,1)},\ V_{(0,2,0)},\ V_{(0,0,2)}.$$
	If $V_\gamma$ is not one of these irreducible representations, then the operator
	$$\left(\slashed{\mathfrak{D}}_{A_\Sigma}^0\right)_\gamma : \Hom\left(V_\gamma, \Delta\otimes\mathfrak{sp}(1)_\mathbb{C}\right)^{\mathfrak{sp}(1)_u\oplus \mathfrak{sp}(1)_d} \to \Hom\left(V_\gamma, \Delta\otimes\mathfrak{sp}(1)_\mathbb{C}\right)^{\mathfrak{sp}(1)_u\oplus \mathfrak{sp}(1)_d}$$
	has no eigenvalues in the interval $\left[-\frac{5}{2}, \frac{5}{2}\right]$.
\end{cor}

\subsection{Eigenvalues of the Untwisted Dirac operator}
Now, we compute the eigenvalues of the untwisted Dirac operator. Although, by Theorem \ref{indjump}, we only need to compute the eigenvalues of the twisted Dirac operator to find the critical weights, the eigenvalues of the untwisted Dirac operator will be needed later to calculate the index of the twisted Dirac operator. We have the main theorem as follows.
\begin{thm}\label{eigdirun}
The eigenvalues of the untwisted Dirac operator $\left(\mathcal{D}_\Sigma^0\right)_\gamma$ corresponding to the irreducible representations $V_\gamma$ listed in Corollary \ref{replist} are
    \begin{enumerate}
        \item $-\frac{7}{2}$ for $V_\gamma = V_{(0,0,0)}$,
        \item $\frac{9}{2}$ for $V_\gamma = V_{(0,0,2)}$,
        \item $\frac{1}{6}(-3-2\sqrt{161}),\ \frac{1}{6}(-3+2\sqrt{161})$ for $V_\gamma = V_{(1,0,0)}$,
        \item $\frac{1}{6}(-3-8\sqrt{11}),\ \frac{1}{6}(-3+8\sqrt{11}),\ \ -\frac{23}{6}$ for $V_\gamma = V_{(0,1,1)},$
        \item $\frac{9}{2},\ -\frac{25}{6}$ for $V_\gamma = V_{(0,2,0)}$.
    \end{enumerate}
\end{thm}
\begin{proof}
    We note that for $V_\gamma = V_{(0,0,1)}\cong W_{(0,1)}, V_{(1,0,1)}\cong W_{(0,1)} \oplus W_{(1,0)} \oplus W_{(1,2)}$ or $V_{(0,1,0)}\cong W_{(1,0)} \oplus W_{(0,1)}$, the space $\Hom\left(V_\gamma, \Delta \otimes W_{(0,2)}\right)^{Sp(1)_u \otimes Sp(1)_d}$ is a $0$-dimensional vector space, by Schur's lemma.
 \begin{itemize}
     \item $V_\gamma = V_{(0,0,0)} \cong W_{(0,0)}$. Then,
\begin{align*}
    \Hom\left(V_\gamma, \Delta\right)^{Sp(1)_u \otimes Sp(1)_d} \cong \Hom\left(W_{(0,0)}, \Delta\right)^{Sp(1)_u \otimes Sp(1)_d}
\end{align*}
is $1$-dimensional. With respect to a basis is given by
$$q^{(0,0)} : V_{(0,0,0)} \to W_{(0,0)} \to \Delta$$
we have
$$\left(\mathcal{D}^0_\Sigma\right)_\gamma = \left(\mathcal{D}^1_\Sigma\right)_\gamma -\frac{1}{2}\phi = -\frac{1}{2}\phi.$$
Then the eigenvalue of $\left(\mathcal{D}^0_\Sigma\right)_\gamma$ is $-\frac{7}{2}$ with multiplicity $1$.
\item $V_\gamma = V_{(0,0,2)} \cong W_{(0,2)}$. Then,
\begin{align*}
    \Hom\left(V_\gamma, \Delta\right)^{Sp(1)_u \otimes Sp(1)_d} \cong \Hom\left(W_{(0,2)}, \Delta\right)^{Sp(1)_u \otimes Sp(1)_d}
\end{align*}
is $1$-dimensional. With respect to a basis is given by
$$q^{(0,2)} : V_{(0,0,2)} \to W_{(0,2)} \to \Delta.$$
we have,
$$\left(\mathcal{D}^t_\Sigma\right)_\gamma = 4 - \frac{t-1}{2}.$$
Then, $\left(\mathcal{D}^{\frac{1}{3}}_\Sigma\right)^2_\gamma = 169/9$, which shows the consistency of the calculation. Finally, for $t = 0$, the eigenvalue is given by $9/2$.
\item $V_\gamma = V_{(1,0,0)} \cong W_{(0,0)} \oplus W_{(1,1)}$. Then,
\begin{align*}
    \Hom\left(V_\gamma, \Delta\right)^{Sp(1)_u \otimes Sp(1)_d} \cong \Hom\left(V_{(1,0,0)}, \Delta\right)^{Sp(1)_u \otimes Sp(1)_d}
\end{align*}
is $2$-dimensional. With respect to a basis given by
\begin{align*}
    q^{(0,0)} : V_{(1,0,0)} \to W_{(0,0)} \to \Delta\\
    q^{(1,1)} : V_{(1,0,0)} \to W_{(1,1)} \to \Delta
\end{align*}
we have,
$$\left(\mathcal{D}^t_\Sigma\right)_\gamma = \begin{pmatrix}
    \frac{7}{2}(t-1) &\frac{8\sqrt{5}}{3}\\
    \frac{2\sqrt{5}}{3} &2+\frac{1-t}{2}
\end{pmatrix}.$$
We note that for $t = 1/3$, we have
$\left(\mathcal{D}^{\frac{1}{3}}_\Sigma\right)^2_\gamma = \diag(43/3, 43/3)$, which shows the consistency of the calculation. Finally, for $t = 0$, the eigenvalues are given by $\frac{1}{6}(-3-2\sqrt{161}), \frac{1}{6}(-3+2\sqrt{161})$. 
\item $V_\gamma = V_{(0,1,1)} \cong W_{(0,0)} \oplus W_{(1,1)} \oplus W_{(0,2)}$. Then,
\begin{align*}
    \Hom\left(V_\gamma, \Delta\right)^{Sp(1)_u \otimes Sp(1)_d} \cong \Hom\left(V_{(0,1,1)}, \Delta\right)^{Sp(1)_u \otimes Sp(1)_d}
\end{align*}
is $3$-dimensional. With respect to a basis given by
\begin{align*}
    q^{(0,0)} : V_{(0,1,1)} \to W_{(0,0)} \to \Delta\\
    q^{(1,1)} : V_{(0,1,1)} \to W_{(1,1)} \to \Delta\\
    q^{(0,2)} : V_{(0,1,1)} \to W_{(0,2)} \to \Delta
\end{align*}
we have,
$$\left(\mathcal{D}^t_\Sigma\right)_\gamma = \begin{pmatrix}
    \frac{7}{2}(t-1) &\frac{4\sqrt{5}}{3} &-5\\
    \frac{\sqrt{5}}{3} &-3+\frac{1-t}{2} &-\sqrt{5}\\
    -\frac{5}{3} &-\frac{4\sqrt{5}}{3} &\frac{2}{3}+\frac{1-t}{2}
\end{pmatrix}.$$
We note that for $t = 1/3$, we have
$\left(\mathcal{D}^{\frac{1}{3}}_\Sigma\right)^2_\gamma = \diag(16,16,16)$, which shows the consistency of the calculation. Finally, for $t = 0$, the eigenvalues are given by $\frac{1}{6}(-3-8\sqrt{11}), \frac{1}{6}(-3+8\sqrt{11}), -\frac{23}{6}$.
\item $V_\gamma = V_{(0,2,0)} \cong W_{(1,1)} \oplus W_{(2,0)} \oplus W_{(0,2)}$. Then,
\begin{align*}
    \Hom\left(V_\gamma, \Delta\right)^{Sp(1)_u \otimes Sp(1)_d} \cong \Hom\left(V_{(0,2,0)}, \Delta\right)^{Sp(1)_u \otimes Sp(1)_d}
\end{align*}
is $2$-dimensional. With respect to a basis given by
\begin{align*}
    q^{(0,2)} : V_{(0,2,0)} \to W_{(0,2)} \to \Delta\\
    q^{(1,1)} : V_{(0,2,0)} \to W_{(1,1)} \to \Delta
\end{align*}
we have,
$$\left(\mathcal{D}^t_\Sigma\right)_\gamma = \begin{pmatrix}
    -\frac{8}{3}+\frac{1-t}{2} &\frac{4\sqrt{10}}{3}\\
    \sqrt{10} &2+\frac{1-t}{2}
\end{pmatrix}.$$
We note that for $t = 1/3$, we have
$\left(\mathcal{D}^{\frac{1}{3}}_\Sigma\right)^2_\gamma = \diag(169/9, 169/9)$, which shows the consistency of the calculation. Finally, for $t = 0$, the eigenvalues are given by $\frac{9}{2},-\frac{25}{6}$.
\end{itemize}
\end{proof}

\subsection{Eigenvalues of the Twisted Dirac Operator}
Now, we compute the eigenvalues of the twisted Dirac operator. The main theorem is given as follows.
\begin{thm}\label{eigdir6}
The eigenvalues of the twisted Dirac operator $\left(\slashed{\mathfrak{D}}_{A_\Sigma}^0\right)_\gamma$ corresponding to the irreducible representations $V_\gamma$ listed in Corollary \ref{replist} are
    \begin{enumerate}
        \item $\frac{1}{2}$ for $V_\gamma = V_{(0,0,0)}$,
        \item $\frac{1}{2}(-1-2\sqrt{17}),\ \frac{1}{2}(-1+2\sqrt{17})$ for $V_\gamma = V_{(0,0,2)}$,
        \item $\frac{19}{6},\ -\frac{17}{6}$ for $V_\gamma = V_{(1,0,0)}$,
        \item $\frac{1}{6}(-3-16\sqrt{2}),\ \frac{1}{6}(-3+16\sqrt{2}),\ \frac{1}{6}(1-8\sqrt{6}),\ \frac{1}{6}(1+8\sqrt{6})$ for $V_\gamma = V_{(0,1,1)}$,
        \item $\frac{1}{2}(-1-2\sqrt{17}),\ \frac{1}{2}(-1+2\sqrt{17}),\ -\frac{7}{2}$ for $V_\gamma = V_{(0,2,0)}$.
    \end{enumerate}
\end{thm}
\begin{proof}
    We note that
\begin{align}
    \Delta \otimes \mathfrak{sp}(1)_{\mathbb{C}} \cong W_{(0,2)} \oplus [W_{(1,1)}\oplus W_{(1,3)}] \oplus [W_{(0,0)}\oplus W_{(0,2)}\oplus W_{(0,4)}].
\end{align}
Now,
$$L^2(\slashed{S}(\Sigma)\otimes \mathfrak{g}_P) = \bigoplus_{\gamma \in \widehat{Sp(2)\times Sp(1)}}\Hom\left(V_\gamma, \Delta \otimes W_{(0,2)}\right)^{Sp(1)_u \otimes Sp(1)_d} \otimes V_\gamma$$
where $\mathfrak{g}_P$ is the adjoint bundle of the $Sp(1)$-principal bundle $P = (Sp(2) \times Sp(1)) \times_\lambda Sp(1)$ over $\Sigma^7$, defined in section \ref{section3}. Hence, for each $\gamma$, the operator $\left(\slashed{\mathfrak{D}}^t_{A_\Sigma}\right)_\gamma$ defined in (\ref{diractgamma}), acts on the space $\Hom\left(V_\gamma, \Delta \otimes W_{(0,2)}\right)^{Sp(1)_u \otimes Sp(1)_d}$.\par 
We note that for $V_\gamma = V_{(0,0,1)}\cong W_{(0,1)}, V_{(1,0,1)}\cong W_{(0,1)} \oplus W_{(1,0)} \oplus W_{(1,2)}$ or $V_{(0,1,0)}\cong W_{(1,0)} \oplus W_{(0,1)}$, the space $\Hom\left(V_\gamma, \Delta \otimes W_{(0,2)}\right)^{Sp(1)_u \otimes Sp(1)_d}$ is a $0$-dimensional vector space, by Schur's lemma.
\begin{itemize}
\item $V_\gamma = V_{(0,0,0)} \cong W_{(0,0)}$. Then,
\begin{align*}
    \Hom\left(V_\gamma, \Delta \otimes W_{(0,2)}\right)^{Sp(1)_u \otimes Sp(1)_d} \cong \Hom\left(W_{(0,0)}, \Delta \otimes W_{(0,2)}\right)^{Sp(1)_u \otimes Sp(1)_d}
\end{align*}
is $1$-dimensional. With respect to a basis is given by
$$q^{(0,0)}_{(0,2)(0,2)} : V_{(0,0,0)} \to W_{(0,0)} \to W_{(0,2)} \otimes W_{(0,2)} \to \Delta \otimes W_{(0,2)}$$
we have,
$$\left(\slashed{\mathfrak{D}}^0_{A_\Sigma}\right)_\gamma = \left(\slashed{\mathfrak{D}}^1_{A_\Sigma}\right)_\gamma -\frac{1}{2}\phi = -\frac{1}{2}\phi.$$
Then the eigenvalue of $\left(\slashed{\mathfrak{D}}_{A_\Sigma}^0\right)_\gamma$ is $\frac{1}{2}$ with multiplicity $1$.
\item $V_\gamma = V_{(0,0,2)} \cong W_{(0,2)}$. Then,
\begin{align*}
    \Hom\left(V_\gamma, \Delta \otimes W_{(0,2)}\right)^{Sp(1)_u \otimes Sp(1)_d} \cong \Hom\left(W_{(0,2)}, \Delta \otimes W_{(0,2)}\right)^{Sp(1)_u \otimes Sp(1)_d}
\end{align*}
is $2$-dimensional. With respect to a basis given by
\begin{align*}
    q^{(0,2)}_{(0,0)(0,2)} : V_{(0,0,2)} \to W_{(0,2)} \to W_{(0,0)} \otimes W_{(0,2)} \to \Delta \otimes W_{(0,2)}\\
    q^{(0,2)}_{(0,2)(0,2)} : V_{(0,0,2)} \to W_{(0,2)} \to W_{(0,2)} \otimes W_{(0,2)} \to \Delta \otimes W_{(0,2)}
\end{align*}
we have,
$$\left(\slashed{\mathfrak{D}}^t_{A_\Sigma}\right)_\gamma = \begin{pmatrix}
    \frac{7}{2}(t-1) &4\\
    2 &2+\frac{1-t}{2}
\end{pmatrix}.$$
We note that for $t = 1/3$, we have
$\left(\slashed{\mathfrak{D}}^{\frac{1}{3}}_{A_\Sigma}\right)^2_\gamma = \diag(121/9, 121/9)$, which shows the consistency of the calculation. Finally, for $t = 0$, the eigenvalues are given by $\frac{1}{2}(-1-2\sqrt{17}),\frac{1}{2}(-1+2\sqrt{17})$. 
\item $V_\gamma = V_{(1,0,0)} \cong W_{(0,0)} \oplus W_{(1,1)}$. Then,
\begin{align*}
    \Hom\left(V_\gamma, \Delta \otimes W_{(0,2)}\right)^{Sp(1)_u \otimes Sp(1)_d} \cong \Hom\left(V_{(1,0,0)}, \Delta \otimes W_{(0,2)}\right)^{Sp(1)_u \otimes Sp(1)_d}
\end{align*}
is $2$-dimensional. With respect to a basis given by
\begin{align*}
    q^{(0,0)}_{(0,2)(0,2)} : V_{(1,0,0)} \to W_{(0,0)} \to W_{(0,2)} \otimes W_{(0,2)} \to \Delta \otimes W_{(0,2)}\\
    q^{(1,1)}_{(0,2)(0,2)} : V_{(1,0,0)} \to W_{(1,1)} \to W_{(1,1)} \otimes W_{(0,2)} \to \Delta \otimes W_{(0,2)}
\end{align*}
we have,
$$\left(\slashed{\mathfrak{D}}^t_{A_\Sigma}\right)_\gamma = \begin{pmatrix}
    \frac{1-t}{2} & \frac{8\sqrt{5}}{3}\\
    \frac{2\sqrt{5}}{3} & -\frac{2}{3}+\frac{1-t}{2}.
\end{pmatrix}.$$
We note that for $t = 1/3$, we have
$\left(\slashed{\mathfrak{D}}^{\frac{1}{3}}_{A_\Sigma}\right)^2_\gamma = \diag(9, 9)$, which shows the consistency of the calculation. Finally, for $t = 0$, the eigenvalues are given by $\frac{19}{6},-\frac{17}{6}$.
\item $V_\gamma = V_{(0,1,1)} \cong W_{(0,0)} \oplus W_{(1,1)} \oplus W_{(0,2)}$. Then,
\begin{align*}
    \Hom\left(V_\gamma, \Delta \otimes W_{(0,2)}\right)^{Sp(1)_u \otimes Sp(1)_d} \cong \Hom\left(V_{(0,1,1)}, \Delta \otimes W_{(0,2)}\right)^{Sp(1)_u \otimes Sp(1)_d}
\end{align*}
is $4$-dimensional. With respect to a basis given by
\begin{align*}
    q^{(0,0)}_{(0,2)(0,2)} : V_{(0,1,1)} \to W_{(0,0)} \to W_{(0,2)} \otimes W_{(0,2)} \to \Delta \otimes W_{(0,2)}\\
    q^{(1,1)}_{(1,1)(0,2)} : V_{(0,1,1)} \to W_{(1,1)} \to W_{(1,1)} \otimes W_{(0,2)} \to \Delta \otimes W_{(0,2)}\\
    q^{(0,2)}_{(0,0)(0,2)} : V_{(0,1,1)} \to W_{(0,2)} \to W_{(0,0)} \otimes W_{(0,2)} \to \Delta \otimes W_{(0,2)}\\
    q^{(0,2)}_{(0,2)(0,2)} : V_{(0,1,1)} \to W_{(0,2)} \to W_{(0,2)} \otimes W_{(0,2)} \to \Delta \otimes W_{(0,2)}
\end{align*}
we have,
$$\left(\slashed{\mathfrak{D}}^t_{A_\Sigma}\right)_\gamma = \begin{pmatrix}
   \frac{1-t}{2} &\frac{4\sqrt{5}}{3} &\frac{5}{3} &\frac{10}{3} \\
   \frac{\sqrt{5}}{3} &1+\frac{1-t}{2} &\frac{\sqrt{5}}{3} &-\frac{2\sqrt{5}}{3} \\
   \frac{5}{3} &\frac{4\sqrt{5}}{3} &\frac{7}{2}(t-1) &\frac{2}{3} \\
   \frac{5}{3} &-\frac{4\sqrt{5}}{3} &\frac{1}{3} &\frac{1}{3}+\frac{1-t}{2} 
\end{pmatrix}.$$
We note that for $t = 1/3$, we have
$\left(\slashed{\mathfrak{D}}^{\frac{1}{3}}_{A_\Sigma}\right)^2_\gamma = \diag(32/3,32/3,32/3,32/3)$, which shows the consistency of the calculation. Finally, for $t = 0$, the eigenvalues are given by $\frac{1}{6}(-3-16\sqrt{2}), \frac{1}{6}(-3+16\sqrt{2}), \frac{1}{6}(1-8\sqrt{6}), \frac{1}{6}(1+8\sqrt{6})$.
\item $V_\gamma = V_{(0,2,0)} \cong W_{(1,1)} \oplus W_{(2,0)} \oplus W_{(0,2)}$. Then,
\begin{align*}
    \Hom\left(V_\gamma, \Delta \otimes W_{(0,2)}\right)^{Sp(1)_u \otimes Sp(1)_d} \cong \Hom\left(V_{(0,1,1)}, \Delta \otimes W_{(0,2)}\right)^{Sp(1)_u \otimes Sp(1)_d}
\end{align*}
is $3$-dimensional. With respect to a basis given by
\begin{align*}
    q^{(1,1)}_{(1,1)(0,2)} : V_{(0,2,0)} \to W_{(1,1)} \to W_{(1,1)} \otimes W_{(0,2)} \to \Delta \otimes W_{(0,2)}\\
    q^{(0,2)}_{(0,0)(0,2)} : V_{(0,2,0)} \to W_{(0,2)} \to W_{(0,0)} \otimes W_{(0,2)} \to \Delta \otimes W_{(0,2)}\\
    q^{(0,2)}_{(0,2)(0,2)} : V_{(0,2,0)} \to W_{(0,2)} \to W_{(0,2)} \otimes W_{(0,2)} \to \Delta \otimes W_{(0,2)}.
\end{align*}
we have,
$$\left(\slashed{\mathfrak{D}}^t_{A_\Sigma}\right)_\gamma = \begin{pmatrix}
   -\frac{2}{3}+\frac{1-t}{2} &-\frac{\sqrt{10}}{3} &\frac{2\sqrt{10}}{3}\\
   -\frac{4\sqrt{10}}{3} &\frac{7}{2}(t-1) & -\frac{8}{3}\\
   \frac{4\sqrt{10}}{3} &-\frac{4}{3} & -\frac{4}{3}+\frac{1-t}{2} 
\end{pmatrix}.$$
We note that for $t = 1/3$, we have
$\left(\slashed{\mathfrak{D}}^{\frac{1}{3}}_{A_\Sigma}\right)^2_\gamma = \diag(121/9,121/9,121/9)$, which shows the consistency of the calculation. Finally, for $t = 0$, the eigenvalues are given by $\frac{1}{2}(-1-2\sqrt{17}),\ \frac{1}{2}(-1+2\sqrt{17}),\ -\frac{7}{2}$. 
\end{itemize}
\end{proof}
\begin{cor}\label{eigdircor6}
    The only eigenvalue of the twisted Dirac operator $\slashed{\mathfrak{D}}_{A_\Sigma}^0$ in the interval $\left[-\frac{5}{2}, \frac{5}{2}\right]$ is $\frac{1}{2}$ corresponding to the trivial representation $V_{(0,0,0)}$.
\end{cor} 

\section{Deformations of Clarke--Oliveira's Instantons}\label{section5}
In this section we calculate the deformations of Clarke--Oliveira's instantons and calculate the virtual dimensions of the moduli spaces, following similar techniques as in \cite{me2023paper}. We compute the index of the twisted Dirac operator $\slashed{\mathfrak{D}}_A^- : W^{k,2}_{-\frac{7}{2}} \to W^{k-1,2}_{-\frac{9}{2}}$ on $\slashed{S}^-(S^4)$, where $A$ is Clarke--Oliveira's instanton (the one parameter family or the limiting instanton), using the Atiyah--Patodi--Singer theorem and various other techniques developed in \cite{me2023paper}, \cite{Atiyah-patodi1975:1}, \cite{Atiyah-patodi1976:2}, \cite{Atiyah-patodi1976:3}, \cite{eguchi1980}, \cite{galkey1981}.

\subsection{Index of the Twisted Dirac Operator}

Let $g_C$ be the Bryant--Salamon metric (\ref{bsmetric}) on $\slashed{S}^-(S^4)$. We define the asymptotically cylindrical ``cigar'' metric $g_{CI} := \frac{1}{\varrho^2}g_C$, where $\varrho$ is the radius function (see \cite{me2023paper}).\par
\begin{figure}[H]
	\centering
	\begin{tikzpicture}
		\draw[dashed] (4,0) ellipse (0.3cm and 1cm);
		\draw (0,1) -- (4,1);
		\draw (0,-1) -- (4,-1);
		\draw (0,1) arc[start angle=90, end angle=270,radius=1cm];
		\draw (-1.8,0) node {$r = 0$};
        \draw (-1,-1) node {$S^4$};
		\draw (4,-1.3) node {$\Sigma^7$};
		\draw (5,0) node {$r \to \infty$};
	\end{tikzpicture}
 \caption{$\slashed{S}^-(S^4)$ with cigar metric.}
	\label{fig:cigar6}
\end{figure}
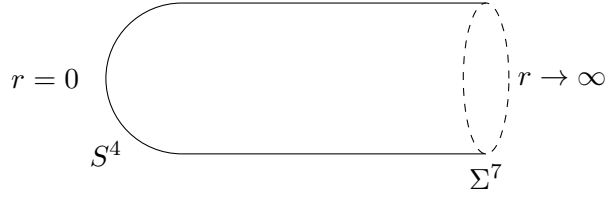\noindent
Then, we have \cite{me2023paper}
\begin{equation}\label{indrate2}
    \ind\left(\slashed{\mathfrak{D}}_{A,C}^- : W^{k,2}_{-\frac{7}{2}} \to W^{k-1,2}_{-\frac{9}{2}}\right) = \ind\left(\slashed{\mathfrak{D}}_{A,{CI}}^- : W^{k,2}_{CI} \to W^{k-1,2}_{CI}\right).
\end{equation}
where $A$ is either $A_{y_0}$ or $A_{\text{lim}}$.\par
Throughout this section, we identify $\mathbb{R} \times \Sigma^7$ with $(0,\infty) \times \Sigma^7$ via $t = \ln r$ for $r \in (0,\infty)$ and $t \in \mathbb{R}$.\par
Define a function $\overline{\varphi} : \mathbb{R} \to \mathbb{R}$ by
\begin{equation}\label{phibar'}
    \overline{\varphi}(t) = \begin{cases}
    -3 &t < -T\\
    a', &-T < t < -\frac{T}{2}\\
    \varphi(t), &-\frac{T}{2} < t < \frac{T}{2}\\
    a, &\frac{T}{2} < t < T\\
    0 &t > T
    \end{cases}
\end{equation}
where $a$ is a smooth interpolation between its values at $\frac{T}{2}$ and $T$ and $a'$ is that of between its values at $-T$ and $-\frac{T}{2}$.\par 
Consider the connection on $\slashed{S}^-(S^4)$ given by
\begin{equation}
    \overline{A} = A_\Sigma + \overline{\varphi}(t)e^aT_a.
\end{equation}
This connection has the same limits at $ t = \pm \infty$ as Clarke--Oliveira's instantons $A$.\par
We note that $\slashed{S}^-(S^4)$ can be considered as the space $[0,\infty) \times (Sp(2) \times Sp(1))/\sim$ where
\begin{align}\label{simlev}
    (r,g) &\sim (r, g\cdot h)\ \text{ for all }r >0,\ h \in Sp(1)^2\nonumber\\
    (0,g) &\sim (0, g\cdot h)\ \text{ for all }h \in Sp(1)^3.
\end{align}
\begin{prop}\label{indexequal2}
	Let $K^8_R := [0,R) \times (Sp(2) \times Sp(1))/\sim$ be a compact $8$-dimensional subset of $\slashed{S}^-(S^4)$, where $R >0$ and $\sim$ is defined as (\ref{simlev}). Then for sufficiently large $R$, we have
	$$\ind\left(\slashed{\mathfrak{D}}_{A,{CI}}^-, \slashed{S}^-(S^4), g_{CI}\right) = \ind\left(\slashed{\mathfrak{D}}_{\overline{A},{CI}}^-, K^8_R, g_{CI}\right),$$
 and for sufficiently large $T$, we have
 $$\ind\left(\slashed{\mathfrak{D}}_A^-, \mathbb{R} \times \Sigma^7, g\right) = \ind\left(\slashed{\mathfrak{D}}_{\overline{A}}^-, [-T,T] \times \Sigma^7, g\right).$$
 where $g$ is the cylindrical metric given by $g = dt^2 +g_{\Sigma^7}$.
\end{prop}
The proof is similar to the proof of proposition 5.1 in \cite{me2023paper}.\par
Now, the index of the twisted Dirac operator $\slashed{\mathfrak{D}}_{\overline{A},{CI}}^-$ on $K^8_R$ is given by the Atiyah--Patodi--Singer index theorem, which states that
\begin{align}\label{APS6.1}
    \ind\left(\slashed{\mathfrak{D}}_{\overline{A},{CI}}^-, K^8_R, g_{CI}\right) = I\left(\slashed{\mathfrak{D}}_{\overline{A},{CI}}^-, K^8_R, g_{CI}\right) + \frac{1}{2}\eta(\slashed{\mathfrak{D}}_{A_\Sigma}, \partial K^8_R).
\end{align}
By proposition \ref{indexequal2}, the index is independent of $R$, hence taking $R \to \infty$, we have,
\begin{align}\label{APS6.2}
    \ind\left(\slashed{\mathfrak{D}}_{\overline{A},{CI}}^-, \slashed{S}^-(S^4), g_{CI}\right) = I\left(\slashed{\mathfrak{D}}_{\overline{A},{CI}}^-, \slashed{S}^-(S^4), g_{CI}\right) + \frac{1}{2}\eta(\slashed{\mathfrak{D}}_{A_\Sigma}, \Sigma^7).
\end{align}
where, the term $I(\slashed{\mathfrak{D}}_{A,{CI}}^-, \slashed{S}^-(S^4), g_{CI})$ in (\ref{APS6.2}) is given by
\begin{align}\label{iterm}
	I(\slashed{\mathfrak{D}}_{A,{CI}}^-, \slashed{S}^-(S^4), g_{CI}) &= - \int_{\slashed{S}^-(S^4)}\widehat{A}(\slashed{S}^-(S^4))\ch(\mathfrak{g}_P \otimes \mathbb{C})\nonumber\\
	&= -\int_{\slashed{S}^-(S^4)}\left(1 - \frac{1}{24}p_1(\slashed{S}^-(S^4)) + \frac{1}{5760}(7p_1(\slashed{S}^-(S^4))^2 - 4p_2(\slashed{S}^-(S^4)))\right)\nonumber\\
 &\hspace{5cm}\left(\dim\mathfrak{g} + p_1(\mathfrak{g}_P) + \frac{1}{12}\left(p_1(\mathfrak{g}_P)^2 - 2p_2(\mathfrak{g}_P)\right)\right)\nonumber\\
	&= -\frac{1}{12}\int_{\slashed{S}^-(S^4)}\left(p_1(\mathfrak{g}_P)^2 - 2p_2(\mathfrak{g}_P)\right) + \frac{1}{24}\int_{\slashed{S}^-(S^4)}p_1(\slashed{S}^-(S^4))p_1(\mathfrak{g}_P)\nonumber\\
    &\hspace{0.4cm} - \frac{1}{5760}\dim\mathfrak{g}\int_{\slashed{S}^-(S^4)}(7p_1(\slashed{S}^-(S^4))^2 - 4p_2(\slashed{S}^-(S^4))),
\end{align}
where $p_i$ denotes the $i$th Pontryagin class.
\subsection{Eta Invariant of the Boundary}
We calculate the eta invariant $\eta(\slashed{\mathfrak{D}}_{A_\Sigma}, \Sigma^7)$ of the twisted Dirac operator by first, relating it to the untwisted Dirac operator on the squashed sphere, and then, relating it to the untwisted Dirac operator on the round sphere, whose eta invariant is known to be zero.\par 
Recall Clarke--Oliveira's instanton (\ref{COinst}) with $\varphi$ given by (\ref{valueofphi}). The instanton can be identified with a family of connections $\{A_t : t \in \mathbb{R}\}$ on $\Sigma^7$, where $t = \ln r$.
The family of Dirac operators twisted by the connections $A_t$ is given by
\begin{align}\label{diracfam}
    \slashed{\mathfrak{D}}_{A_{t, \Sigma}} &= \slashed{\mathfrak{D}}_{A_\Sigma} + \varphi(t)e^aT_a\\
    &= \slashed{\mathfrak{D}}_{A_\Sigma} + \frac{1}{3}\varphi(t)\widehat{e}^aT_a,
\end{align}
where $\varphi(t)$ varies from $-3$ to $0$. Then, from (\ref{flatcon}), (\ref{diracfam}) and remark \ref{untflat}, we note that, for $\varphi(t) = 0$, we have the Dirac operator $\slashed{\mathfrak{D}}_{A_\Sigma}$ twisted by the canonical connection, and for $\varphi(t) = -3$, we have the Dirac operator twisted by the flat connection, given by
\begin{align}\label{flatcan}
    \mathcal{D}_{\Sigma} = \slashed{\mathfrak{D}}_{A_\Sigma} -\widehat{e}^aT_a.
\end{align}
We identify the family of Dirac operators $\left\{\slashed{\mathfrak{D}}_{A_{t, \Sigma}}\right\}_{t \in \mathbb{R}}$ on $\Sigma^7$ with the Dirac operator $\slashed{\mathfrak{D}}_A^-$ on the cylinder $\slashed{C} := \mathbb{R} \times \Sigma^7$, given by
$$\slashed{\mathfrak{D}}_A^- = dt \cdot \left(\frac{d}{dt} - \slashed{\mathfrak{D}}_{A_{t, \Sigma}}\right).$$
\begin{figure}[ht]
	\centering
	\begin{tikzpicture}
		\draw (0,0) ellipse (0.3cm and 1cm);   
		\draw[dashed] (4,0) ellipse (0.3cm and 1cm);
		\draw (0,1) -- (4,1);
		\draw (0,-1) -- (4,-1);
		\draw (0,-1.3) node {$\Sigma^7$};
        \draw (0,-2) node {$A_0$};
		\draw (4,-1.3) node {$\Sigma^7$};
        \draw (4,-2) node {$A_\Sigma$};
        \draw (2,-1.7) node {$\slashed{C}_A$};
	\end{tikzpicture}
 \caption{The cylinder $\slashed{C}$.}
	\label{fig:cigar6.1}
\end{figure}\par
Then, from \cite{kronheimer-mrowka}, the index of the Dirac operator $\slashed{\mathfrak{D}}_A^-$ on $\mathbb{R} \times \Sigma^7$ is
\begin{align}\label{indconcone}
    \Ind (\slashed{\mathfrak{D}}_A^-, \slashed{C}) = -\Sf\left(\left\{\slashed{\mathfrak{D}}_{A_{t, \Sigma}}\right\}_{t \in \mathbb{R}}\right).
\end{align}
Now, from Proposition \ref{indexequal2} and applying the Atiyah--Patodi--Singer index formula on $[-T, T] \times \Sigma^7$, we have
$$\Ind (\slashed{\mathfrak{D}}_A^-, \slashed{C}) = \Ind (\slashed{\mathfrak{D}}_{\overline{A}}^-, [-T, T] \times \Sigma^7) = I\left(\slashed{\mathfrak{D}}_{\overline{A}}^-, [-T, T] \times \Sigma^7\right) + \frac{1}{2}\eta(\partial([-T, T] \times \Sigma^7))$$
where the term $\eta(\partial([-T, T] \times \Sigma^7))$ is the eta invariant for the operator $\slashed{\mathfrak{D}}_{\overline{A}}^-$ restricted to the submanifold $\partial([-T, T] \times \Sigma^7)$. By proposition \ref{indexequal2}, since $\Ind (\slashed{\mathfrak{D}}_A^-, \slashed{C})$ is independent of $T$, taking $T \to \infty$, we get,
$$\Ind (\slashed{\mathfrak{D}}_A^-, \slashed{C}) = I\left(\slashed{\mathfrak{D}}_{\overline{A}}^-, \slashed{C}\right) + \frac{1}{2}\eta(\partial\slashed{C})$$
where $\eta(\partial\slashed{C})$ is the eta invariant for $\slashed{\mathfrak{D}}_{\overline{A}}^-$ restricted to the submanifold $\partial\slashed{C}$.\par 
Now, from (\ref{indexequal2}), we have
$$I\left(\slashed{\mathfrak{D}}_{\overline{A}}^-, \slashed{C}\right)  = I\left(\slashed{\mathfrak{D}}_A^-, \mathbb{R} \times \Sigma^7\right)$$
Moreover, since $\partial\slashed{C} = \Sigma^7 \amalg \overline{\Sigma^7}$, where $\overline{\Sigma^7}$ is $\Sigma^7$ with opposite orientation, from (\ref{flatcan}) we have
$$\eta(\partial\slashed{C}) = \eta(\mathcal{D}_{\Sigma}, \overline{\Sigma^7}) + \eta(\slashed{\mathfrak{D}}_{A_\Sigma}, \Sigma^7) = \eta(\slashed{\mathfrak{D}}_{A_\Sigma}, \Sigma^7) - \eta(\mathcal{D}_{\Sigma}, \Sigma^7),$$
So, finally, we have
\begin{align}\label{etainv6.2}
	\frac{1}{2}\eta(\slashed{\mathfrak{D}}_{A_\Sigma}, \Sigma^7) &= \frac{1}{2}\eta(\partial\slashed{C}) + \frac{1}{2}\eta(\mathcal{D}_{\Sigma}, \Sigma^7)\nonumber\\
    &= \Ind (\slashed{\mathfrak{D}}_A^-, \slashed{C}) - I(\slashed{\mathfrak{D}}_A^-, \slashed{C}) +\frac{1}{2}\eta(\mathcal{D}_{\Sigma}, \Sigma^7)\nonumber\\
    &=  -\Sf\left(\left\{\slashed{\mathfrak{D}}_{A_{t, \Sigma}}\right\}_{t \in \mathbb{R}}\right) - I(\slashed{\mathfrak{D}}_A^-, \slashed{C}) +\frac{1}{2}\eta(\mathcal{D}_{\Sigma}, \Sigma^7).
\end{align}

Now, we note that since the squashed sphere does not have an orientation reversing isometry, the eta-invariant of the untwisted Dirac operator $\mathcal{D}_\Sigma$ may not be zero. So, to find the eta-invariant of the untwisted Dirac operator on the squashed sphere, we relate it to that of the round sphere for which we know the eta-invariant of the untwisted Dirac operator is zero.\par
Consider the cylinder $C_\Sigma := \Sigma \times \mathbb{R}$, and for $t \in \mathbb{R}$, a family of Riemannian manifolds $(\Sigma, g_t)$ where for $t = -\infty$ we have the squashed $7$-sphere $\Sigma^7$ and for $t = \infty$ we have the round $7$-sphere $S^7$. \begin{figure}[H]
	\centering
	\begin{tikzpicture}
		\draw [dashed] (1.2,0) ellipse (0.3cm and 2cm);   
		\draw [dashed] (5.2,0) ellipse (0.3cm and 1cm);
		\draw (1.2,2) .. controls (3.5,2) and (4,1) .. (5.2,1);
		\draw (1.2,-2) .. controls (3.5,-2) and (4,-1) .. (5.2,-1);
		\draw (1.2,-2.3) node {$S^7$};
		\draw (5.2,-1.3) node {$\Sigma^7$};
        \draw (3.4,-2.3) node {$C_\Sigma$};
	\end{tikzpicture}
 \caption{The cylinder $C_\Sigma$.}
	\label{fig:cigar6.2}
\end{figure}
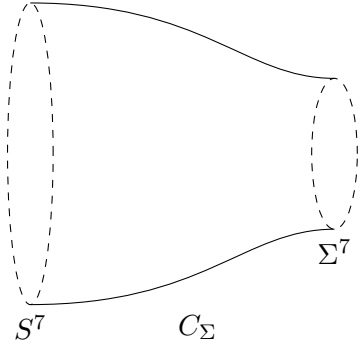\noindent
Consider corresponding family of untwisted Dirac operators $\{\mathcal{D}_{\Sigma,t}\}_{t \in \mathbb{R}}$ which we can identify with an untwisted Dirac operator $\mathcal{D}^-$ on the cylinder $C_\Sigma$. Then, using the result from \cite{kronheimer-mrowka}, we have that the index of the Dirac operator on the cylinder $C_\Sigma$ is the negative of the spectral flow of the family $\{\mathcal{D}_{\Sigma,t}\}_{t \in \mathbb{R}}$, i.e.,
\begin{align}
    \Ind(\mathcal{D}^-, C_\Sigma) = -\Sf\left(\left\{\mathcal{D}_{\Sigma,t}\right\}_{t \in \mathbb{R}}\right).
\end{align}

Now, from Proposition \ref{indexequal2} applying the Atiyah--Patodi--Singer index formula on $[-T, T] \times \Sigma^7$, we have
$$\Ind (\mathcal{D}^-, C_\Sigma) = I\left(\mathcal{D}^-, [-T, T] \times \Sigma^7\right) + \frac{1}{2}\eta(\mathcal{D}_\Sigma, \partial([-T, T] \times \Sigma^7)).$$
Since $\Ind (\mathcal{D}^-, C_\Sigma)$ is independent of $T$, taking $T \to \infty$, we get,
$$\Ind (\mathcal{D}^-, C_\Sigma) = I\left(\mathcal{D}^-, C_\Sigma\right) + \frac{1}{2}\eta(\partial(C_\Sigma)).$$
Moreover, since $\partial(C_\Sigma) = \Sigma^7 \amalg \overline{S^7}$, where $\overline{S^7}$ is $S^7$ with opposite orientation, we have
$$\eta(\partial(C_\Sigma)) = \eta(\mathcal{D}_\Sigma, \Sigma^7) - \eta(\mathcal{D}_{\Sigma}, S^7).$$
Hence, we have
\begin{align}\label{etainv6.3}
	\frac{1}{2}\eta(\mathcal{D}_\Sigma, \Sigma^7) &= \frac{1}{2}\eta(\partial(C_\Sigma)) + \frac{1}{2}\eta(\mathcal{D}_{\Sigma}, S^7)\nonumber\\
    &= \Ind(\mathcal{D}^-, C_\Sigma) - I(\mathcal{D}^-, C_\Sigma) + \frac{1}{2}\eta(\mathcal{D}_{\Sigma}, S^7)\nonumber\\
    &=  -\Sf\left(\left\{\mathcal{D}_{\Sigma,t}\right\}_{t \in \mathbb{R}}\right)- I(\mathcal{D}^-, C_\Sigma) + \frac{1}{2}\eta(\mathcal{D}_{\Sigma}, S^7).
\end{align}
Then, substituting (\ref{etainv6.3}) in (\ref{etainv6.2}) and using the fact that $\eta(\mathcal{D}_{\Sigma}, S^7) = 0$, we have 
\begin{align}\label{etainv6.4}
    \frac{1}{2}\eta(\slashed{\mathfrak{D}}_{A_\Sigma}, \Sigma^7) =  -\Sf\left(\left\{\slashed{\mathfrak{D}}_{A_{t, \Sigma}}\right\}_{t \in \mathbb{R}}\right) - I(\slashed{\mathfrak{D}}_A^-, \slashed{C}) -\Sf\left(\left\{\mathcal{D}_{\Sigma,t}\right\}_{t \in \mathbb{R}}\right)- I(\mathcal{D}^-, C_\Sigma).
\end{align}

\subsubsection*{Spectral Flow of the Connection}
We want to find the spectral flow of the family of Dirac operators (\ref{diracfam}). First we compute the eigenvalues of the operator $\widehat{e}^aT_a$ which acts fibre-wise, on $\Delta \otimes \mathfrak{sp}(1)$. Let $e^\mu,\ \mu = 0, 1, \dots, 7$ be a basis of $\Delta$ and $T_a$ is a basis of $\mathfrak{sp}(1)$, for $a = 1,2,3$. Then,
\begin{align*}
	&\ \ \left(\widehat{e}^aT_a\right)\left(\widehat{e}^\mu \otimes T_b\right)\\
	&= (\widehat{e}^a \cdot \widehat{e}^\mu) \otimes [T_a, T_b]\\
	&= (\widehat{E}^a \otimes \ad T_a)(\widehat{e}^\mu \otimes T_b),
\end{align*}
where $\widehat{E}^a$ is the matrix given by Clifford multiplication with $\widehat{e}^a$. Similarly, we calculate the matrices $\ad T_a$. We get the matrix of $\widehat{e}^aT_a$ by taking the Kronecker product of $\widehat{E}^a$ and $\ad T_a$. The eigenvalues of $\widehat{e}^aT_a$ are given in table \ref{tab:eigeneiti}.
\begin{table}[ht]
    \centering
	\renewcommand{\arraystretch}{1.2}
	\begin{tabular}{||c | c ||} 
		\hline
		Eigenvalues & Multiplicities  \\
		\hline\hline
		$4$ & $4$ \\
		\hline
		$-4$ & $4$ \\
		\hline
		$-2$ & $8$ \\
		\hline
		$2$ & $8$ \\
		\hline
	\end{tabular}
 \caption{Eigenvalues of $\widehat{e}^aT_a$ and corresponding multiplicities.}
	\label{tab:eigeneiti}
\end{table}\par
Now, the the highest magnitude among the eigenvalues of $\widehat{e}^aT_a$ is $4$. Hence, from (\ref{flatcan}), any flow from an eigenvalue of the twisted Dirac operator $\slashed{\mathfrak{D}}_{A_\Sigma}$ to an eigenvalue of the untwisted Dirac operator $\mathcal{D}_{\Sigma}$ can have a magnitude of maximum $4$. Moreover, it is evident that the only possible non-zero spectral flow would be a flow from the eigenvalue $1/2$ of $\slashed{\mathfrak{D}}_{A_\Sigma}$ to the eigenvalue $-7/2$ of $\mathcal{D}_{\Sigma}$, since any other flow of eigenvalues of $\slashed{\mathfrak{D}}_{A_\Sigma}$ to eigenvalues of $\mathcal{D}_{\Sigma}$ of opposite sign has magnitude greater than $4$. Now, we recall that $1/2$ is an eigenvalue of $\slashed{\mathfrak{D}}_{A_\Sigma}$ that corresponds to the trivial representation $V_{(0,0,0)}$ of $Sp(2) \times Sp(1)$. Hence, in the decomposition (\ref{frob2}), the eigenspinor $\eta$ corresponding to eigenvalue $1/2$ belongs to the space $\Hom(V_{(0,0,0)}, \Delta \otimes \mathfrak{sp}(1)_\mathbb{C})^{\mathfrak{sp}(1)_u\oplus\mathfrak{sp}(1)_d} \otimes V_{(0,0,0)} \subset L^2(Sp(2) \times Sp(1), \Delta \otimes \mathfrak{sp}(1)_\mathbb{C})^{\mathfrak{sp}(1)_u\oplus\mathfrak{sp}(1)_d}$. Then, from the decomposition
$$\Delta \otimes \mathfrak{sp}(1)_\mathbb{C} \cong W_{(0,0)}\oplus 2W_{(0,2)}\oplus W_{(1,1)}\oplus W_{(1,3)} \oplus  W_{(0,4)},$$
and by Schur's lemma, we have that $\eta \in \Hom(V_{(0,0,0)}, W_{(0,0)})^{\mathfrak{sp}(1)_u\oplus\mathfrak{sp}(1)_d} \otimes V_{(0,0,0)}$ which is a subspace of $L^2(Sp(2) \times Sp(1), \Delta \otimes \mathfrak{sp}(1)_\mathbb{C})^{\mathfrak{sp}(1)_u\oplus\mathfrak{sp}(1)_d}$. Hence, we compute the eigenvalue of $\widehat{e}^aT_a$ corresponding to the trivial subrepresentation $W_{(0,0)}$ of $\Delta \otimes \mathfrak{sp}(1)_\mathbb{C}$.\par
A direct calculation shows that the eigenvalue is $-4$. Thus we have a flow of the eigenvalue moving up to $9/2$ and not down to $-7/2$. Thus, the spectral flow of the family $\left\{\slashed{\mathfrak{D}}_{A_{t, \Sigma}}\right\}_{t \in \mathbb{R}}$ is given by
\begin{equation}\label{spflow2}
    \Sf\left(\left\{\slashed{\mathfrak{D}}_{A_{t, \Sigma}}\right\}_{t \in \mathbb{R}}\right) = 0.
\end{equation}

\subsubsection*{Spectral Flow of the Metric}

Now, we want to find the spectral flow of the family of Dirac operators $\{\mathcal{D}_{\Sigma,t}\}_{t \in \mathbb{R}}$. Consider the Lichnerowicz--Weitzenb\"ock formula for the family of Riemannian manifolds $(\Sigma, g_t)$, given by
\begin{align}
    \left(\mathcal{D}_{\Sigma,t}\right)^2 = \nabla_{\Sigma,t}^*\nabla_{\Sigma,t} + \frac{1}{4}s_t,
\end{align}
where $s_t$ is the scalar curvature of the Riemannian manifold $(\Sigma, g_t)$. Then, for the family of Dirac operators to have a nonzero spectral flow, the Dirac operator $\mathcal{D}_{\Sigma,t}$ should have a zero eigenvalue for some $t$, and it is only possible if the corresponding scalar curvature $s_t$ is zero.\par 
Following \cite{kennon2023}, we consider a family of metrics on $\Sigma$, given by
\begin{align}
    g(t) := a(t)^2(\eta_1^2 + \eta_2^2) + b(t)^2\eta_3^2 + c(t)^2\pi^*g_{S^4}
\end{align}
where $t \in \mathbb{R}$, $\pi : S^7 \to S^4$ is a Riemannian submersion, $\eta_1, \eta_2, \eta_3$ are one forms on $S^7$. Then, for this family, $a = b = c = 1$ corresponds to the round metric, and $a = b = \frac{1}{\sqrt{5}}, c = 1$ corresponds to the squashed metric.\par 
\begin{lem}\cite{kennon2023}
    The Ricci curvature $\Ric(t)$ of the family $g(t)$ is given by
    \begin{align}\label{ric}
        2\left(2 - \frac{b^2}{a^2} + \frac{2a^4}{c^4}\right)(\eta_1^2 + \eta_2^2) + 2\left(\frac{b^4}{a^4} + \frac{2b^4}{c^4}\right)\eta_3^2 + 2\left(6 - \frac{2a^2+b^2}{c^2}\right)\pi^*g_{S^4}.
    \end{align}
\end{lem}
It is interesting to note that the Ricci flow for the family $g(t)$ is well defined, and the only two critical points correspond to the round and squashed metrics respectively \cite{kennon2023}.\par
We can easily calculate the scalar curvature of this family to be
\begin{align*}
    &\ \ 4\left(2 - \frac{b^2}{a^2} + \frac{2a^4}{c^4}\right)\frac{1}{a^2} + 2\left(\frac{b^4}{a^4} + \frac{2b^4}{c^4}\right)\frac{1}{b^2} + 8\left(6 - \frac{2a^2+b^2}{c^2}\right)\frac{1}{c^2}\\
    &= \frac{1}{a^4}(8a^2 - 2b^2) + \frac{1}{c^4}(48c^2 - 8a^2 - 4b^2).
\end{align*}
For $a = b = c = 1$, we have the scalar curvature of the round metric, given by $42$; and for $a = b = \frac{1}{\sqrt{5}}, c = 1$, we have the scalar curvature of the squashed metric, given by $378/5$.\par
Now, we consider a simpler family of Riemannian metrics \begin{align}
    \widetilde{g}(t) := a(t)^2(\eta_1^2 + \eta_2^2 + \eta_3^2) + \pi^*g_{S^4}
\end{align}
where $a(-\infty) = 1$ and $a(\infty) = 1/5$. The corresponding family of scalar curvatures is given by $\frac{6}{a^2} + 48 - 12a^2$, which we note to be always nonzero positive for $a \in \left[\frac{1}{\sqrt{5}}, 1\right]$.\par
Then, since a spectral flow of metrics, from the round sphere to the squashed sphere, being a topological invariant, does not depend on the path, proves that
\begin{equation}\label{spflowmetric}
    \Sf\left(\left\{\mathcal{D}_{\Sigma,t}\right\}_{t \in \mathbb{R}}\right) = 0.
\end{equation}

\subsection{Index of the Dirac Operator Twisted by the One Parameter Family of Instantons}
From (\ref{spflow2}), (\ref{spflowmetric}) and (\ref{etainv6.4}), we have the second term in (\ref{APS6.2}), given by
\begin{align}\label{etainv6.5}
    \frac{1}{2}\eta(\slashed{\mathfrak{D}}_{A_\Sigma}, \Sigma^7) = - I(\slashed{\mathfrak{D}}_A^-, \slashed{C}) - I(\mathcal{D}^-, C_\Sigma).
\end{align}
where $\slashed{\mathfrak{D}}_A^-$ is the Dirac operator on the cylinder $\slashed{C}$ with squashed spheres as boundaries, and $\mathcal{D}^-$ is the Dirac operator on the cylinder $C_\Sigma$ with squashed and round spheres as boundaries. To calculate first term, $I(\slashed{\mathfrak{D}}_{\overline{A},{CI}}^-, \slashed{S}^-(S^4), g_{CI})$ in (\ref{APS6.2}), we split $\slashed{S}^-(S^4)$ in two parts, $\slashed{C}$, and the complement $B_\Sigma$.\par
\begin{figure}[H]
	\centering
	\begin{tikzpicture}
		\draw[dashed] (4,0) ellipse (0.3cm and 1cm);
        \draw (0,0) ellipse (0.3cm and 1cm);
		\draw (0,1) -- (4,1);
		\draw (0,-1) -- (4,-1);
		\draw (0,1) arc[start angle=90, end angle=270,radius=1cm];
		\draw (-1.8,0) node {$r = 0$};
        \draw (-1,-1) node {$S^4$};
		\draw (4,-1.3) node {$\Sigma^7$};
        \draw (0,-1.3) node {$\Sigma^7$};
        \draw (-1,-1.9) node {$A_0$};
		\draw (0,-1.9) node {$A_0$};
        \draw (4,-1.9) node {$A_\Sigma$};
        \draw (2,-2.4) node {$\slashed{C}$};
        \draw (-0.5,-2.4) node {$B_\Sigma$};
		\draw (5,0) node {$r \to \infty$};
	\end{tikzpicture}
 \caption{$\slashed{S}^-(S^4) = B_\Sigma \amalg \slashed{C}$.}
	\label{fig:cigarsplit}
\end{figure}\noindent
Hence,
\begin{align}\label{isplit}
    I(\slashed{\mathfrak{D}}_{\overline{A},{CI}}^-, \slashed{S}^-(S^4), g_{CI}) = I(\mathcal{D}^-,B_\Sigma) + I(\slashed{\mathfrak{D}}_A^-, \slashed{C}).
\end{align}
Then, from (\ref{APS6.2}), (\ref{isplit}) and (\ref{etainv6.5}), we have,
\begin{align}\label{indtom}
    \Ind (\slashed{\mathfrak{D}}_{\overline{A},{CI}}^-, \slashed{S}^-(S^4), g_{CI}) &= I(\mathcal{D}^-,B_\Sigma) - I(\mathcal{D}^-, C_\Sigma)\nonumber\\
    &= I(\mathcal{D}^-,B_\Sigma) + I(\mathcal{D}^-, \overline{C}_\Sigma)\nonumber\\
    &= I(\mathcal{D}^-,B_\Sigma \amalg \overline{C}_\Sigma)\nonumber\\
    &= I(\mathcal{D}^-,M_\Sigma),
\end{align}
where $\overline{C}_\Sigma$ is $C_\Sigma$ with opposite orientation, and the manifold $M_\Sigma$ is diffeomorphic to $B_\Sigma \amalg \overline{C}_\Sigma$.
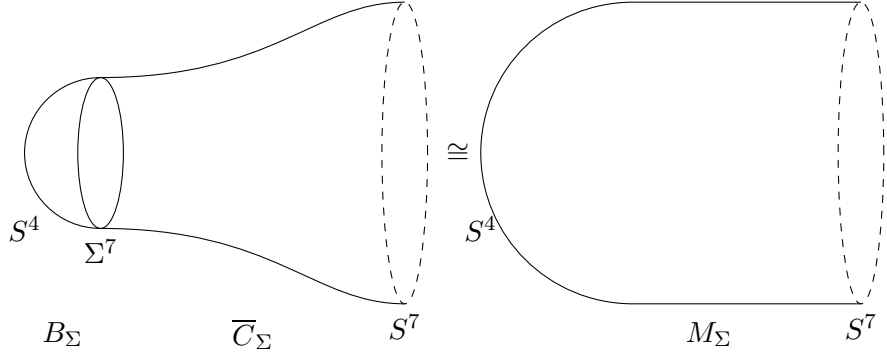
\begin{figure}[H]
	\centering
	\begin{tikzpicture}
        \draw (0,0) ellipse (0.3cm and 1cm);
		\draw (0,1) arc[start angle=90, end angle=270,radius=1cm];
        \draw [dashed] (4,0) ellipse (0.3cm and 2cm);   
		\draw (0,1) .. controls (2.3,1) and (2.8,2) .. (4,2);
		\draw (0,-1) .. controls (2.3,-1) and (2.8,-2) .. (4,-2);
		\draw (4,-2.3) node {$S^7$};;
        \draw (-1,-1) node {$S^4$};
        \draw (0,-1.3) node {$\Sigma^7$};
        \draw (2,-2.4) node {$\overline{C}_\Sigma$};
        \draw (-0.5,-2.4) node {$B_\Sigma$};
        \draw (4.7,0) node {$\cong$};
		\draw (7,2) arc[start angle=90, end angle=270,radius=2cm];
        \draw [dashed] (10,0) ellipse (0.3cm and 2cm);   
		\draw (7,2) -- (10,2);
		\draw (7,-2) -- (10,-2);
		\draw (10,-2.3) node {$S^7$};;
        \draw (5,-1) node {$S^4$};
        \draw (8,-2.4) node {$M_\Sigma$};
	\end{tikzpicture}
 \caption{$B_\Sigma \amalg \overline{C}_\Sigma \cong M_\Sigma$.}
	\label{fig:indcylin}
\end{figure}\noindent
Consider the $8$-dimensional ball $D^8 := \{x \in \mathbb{R}^8 : |x| \leq R \text{ for some }R>0\}$ equipped with the cigar metric $g_{CI}$. Then, it can be easily proved that
\begin{align}\label{connsum1}
    M_\Sigma \#_\partial D^8 \cong \mathbb{H}P^2.
\end{align}
where the boundary gluing is defined by $X \#_\partial Y := X \amalg Y /{\partial X \sim \partial Y}$.
\begin{figure}[H]
	\centering
	\begin{tikzpicture}
		\draw (2,2) arc[start angle=90, end angle=270,radius=2cm];
        \draw (5,0) ellipse (0.3cm and 2cm);   
		\draw (2,2) -- (5,2);
		\draw (2,-2) -- (5,-2);
        \draw (5,-2) arc[start angle=-90, end angle=90,radius=2cm];
		\draw (5,-2.3) node {$S^7$};;
        \draw (-0.5,0) node {$S^4$};
        \draw (7.5,0) node {$\{0\}$};
        \draw (7,0) node {$\bullet$};
        \draw (0,0) node {$\bullet$};
        \draw (5.8,0) node {$D^8$};
        \draw (3.6,0) node {$M_\Sigma$};
	\end{tikzpicture}
 \caption{$M_\Sigma \#_\partial D^8 \cong \mathbb{H}P^2$.}
	\label{fig:HP2}
\end{figure}
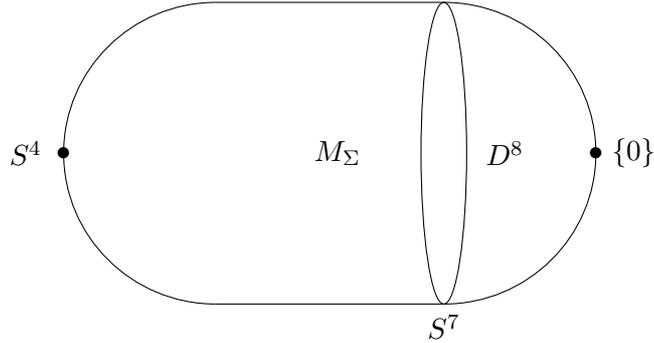\noindent
Let $a \in H^4(\mathbb{H}P^2, \mathbb{Z})$ be a generator such that $\displaystyle \int_{\mathbb{H}P^2} a^2 = 1$. Then, (see \cite{milnor1974char})
$$p_1(\mathbb{H}P^2) = 2a,\ \ p_2(\mathbb{H}P^2) = 7a^2.$$
Hence,
\begin{align*}
    \Ind(\mathcal{D}^-, \mathbb{H}P^2) &= I(\mathcal{D}^-, \mathbb{H}P^2)\\
    &= - \frac{1}{5760}\dim\mathfrak{g}\int_{\mathbb{H}P^2}(7p_1(\mathbb{H}P^2)^2 - 4p_2(\mathbb{H}P^2))\\
    &= 0,
\end{align*}
where the formula for the index is obtained from (\ref{iterm}) by substituting $p_1(\mathfrak{g}_P) = p_2(\mathfrak{g}_P) = 0$, for the trivial connection. Hence,
\begin{align*}
    I(\mathcal{D}^-, \mathbb{H}P^2) = I(\mathcal{D}^-, M_\Sigma) + I(\mathcal{D}_\Sigma^-, D^8) = 0.
\end{align*}
But, from the calculations in \cite{me2023paper}, we know that $p_1(D^8) = p_2(D^8) = 0$, which implies $I(\mathcal{D}^-, D^8) = 0$. Hence, 
\begin{align}\label{im}
    I(\mathcal{D}^-, M_\Sigma) = 0.
\end{align}
From (\ref{indtom}), (\ref{im}) and (\ref{indrate2}), we have
\begin{align}\label{ind-5/2,2}
    \Ind_{-\frac{5}{2}} (\slashed{\mathfrak{D}}_{A_{y_0}}^-, \slashed{S}^-(S^4), g) = \Ind (\slashed{\mathfrak{D}}_{A_{y_0},{CI}}^-, \slashed{S}^-(S^4), g_{CI}) = 0.
\end{align}
Hence, the index of the twisted Dirac operator twisted by the one parameter family of instanton $A_{y_0}$ corresponding to the rate $-5/2$ is $0$.

\subsection{Index of the Dirac Operator Twisted by the Limiting Instantons}
Now, we want to compute the index of the twisted Dirac operator twisted by the limiting instanton $A_{\text{lim}}$. The index is given by
\begin{align}\label{APS7.2}
    &\ \ \ind\left(\slashed{\mathfrak{D}}_{A_{\text{lim}},{CI}}^-, \slashed{S}^-(S^4), g_{CI}\right)\nonumber\\
    &= I\left(\slashed{\mathfrak{D}}_{A_{\text{lim}},{CI}}^-, \slashed{S}^-(S^4), g_{CI}\right) + \frac{1}{2}\eta(\slashed{\mathfrak{D}}_{A_\Sigma}, \Sigma^7)\nonumber\\
    &= I\left(\slashed{\mathfrak{D}}_{A_{\text{lim}},{CI}}^-, \slashed{S}^-(S^4), g_{CI}\right)\nonumber\\
    &\ \ \ \ + \ind\left(\slashed{\mathfrak{D}}_{A_{y_0},{CI}}^-, \slashed{S}^-(S^4), g_{CI}\right) - I\left(\slashed{\mathfrak{D}}_{A_{y_0},{CI}}^-, \slashed{S}^-(S^4), g_{CI}\right)\nonumber\\
    &= I\left(\slashed{\mathfrak{D}}_{A_{\text{lim}},{CI}}^-, \slashed{S}^-(S^4), g_{CI}\right) - I\left(\slashed{\mathfrak{D}}_{A_{y_0},{CI}}^-, \slashed{S}^-(S^4), g_{CI}\right)\nonumber\\
    &= I\left(\slashed{\mathfrak{D}}_{\widetilde{A},{CI}}^-, \mathbb{H}P^2_{(1)} \#\ \overline{\mathbb{H}P^2_{(2)}}, g_{CI}\right)\nonumber\\
    &= -\frac{1}{12}\int\limits_{\mathbb{H}P^2_{(1)} \#\ \overline{\mathbb{H}P^2}_{(2)}}\left(p_1(\mathfrak{g}_{\widetilde{P}})^2 - 2p_2(\mathfrak{g}_{\widetilde{P}})\right)+ \frac{1}{24}\int\limits_{\mathbb{H}P^2_{(1)} \#\ \overline{\mathbb{H}P^2}_{(2)}}p_1(\mathbb{H}P^2_{(1)} \#\ \overline{\mathbb{H}P^2}_{(2)})p_1(\mathfrak{g}_{\widetilde{P}})
\end{align}
where, we have used (\ref{connsum1}) to obtain $\slashed{S}^-(S^4) \#_\partial \overline{\slashed{S}^-(S^4)} \cong \mathbb{H}P^2 \#\ \overline{\mathbb{H}P^2}$. For convenience, we have introduced the notations $\slashed{S}^-(S^4)_{(1)} \#_\partial \overline{\slashed{S}^-(S^4)}_{(2)} \cong \mathbb{H}P^2_{(1)} \#\ \overline{\mathbb{H}P^2}_{(2)}$ to distinguish the two different copies in the connected sum. Here the connection $\widetilde{A}$ on the bundle $\mathfrak{g}_{\widetilde{P}}$ over $\mathbb{H}P^2_{(1)} \#\ \overline{\mathbb{H}P^2_{(2)}}$ is $A_{y_0}$ glued to $A_{\text{lim}}$. 
\begin{figure}[H]
	\centering
	\begin{tikzpicture}
		\draw[dashed] (2,0) ellipse (0.3cm and 1cm);
		\draw (0,1) -- (2,1);
		\draw (0,-1) -- (2,-1);
		\draw (0,1) arc[start angle=90, end angle=270,radius=1cm];

        \draw[dashed] (3,0) ellipse (0.3cm and 1cm);
		\draw (3,1) -- (5,1);
		\draw (3,-1) -- (5,-1);
		\draw (5,1) arc[start angle=90, end angle=-90,radius=1cm];

        \draw (1,1.5) node {$A_{\text{lim}}$};
		\draw (4,1.5) node {$A_{y_0}$};
        \draw (1,-1.5) node {$\slashed{S}^-(S^4)$};
        \draw (4,-1.5) node {$\overline{\slashed{S}^-(S^4)}$};
	\end{tikzpicture}
    \hspace{1cm}
    \begin{tikzpicture}
		\draw (0,1) -- (2,1);
		\draw (0,-1) -- (2,-1);
		\draw (0,1) arc[start angle=90, end angle=270,radius=1cm];

		\draw (2,1) -- (4,1);
		\draw (2,-1) -- (4,-1);
		\draw (4,1) arc[start angle=90, end angle=-90,radius=1cm];
        \draw[dotted] (2,0) ellipse (0.3cm and 1cm);

        \draw (1,1.5) node {$A_{\text{lim}}$};
		\draw (3,1.5) node {$A_{y_0}$};
        \draw (2,-1.5) node {$\mathbb{H}P^2 \#\ \overline{\mathbb{H}P^2}$};

	\end{tikzpicture}
 \caption{$\slashed{S}^-(S^4) \#_\partial \overline{\slashed{S}^-(S^4)} \cong \mathbb{H}P^2 \#\ \overline{\mathbb{H}P^2}$.}
	\label{fig:hp2sumhp2}
\end{figure}
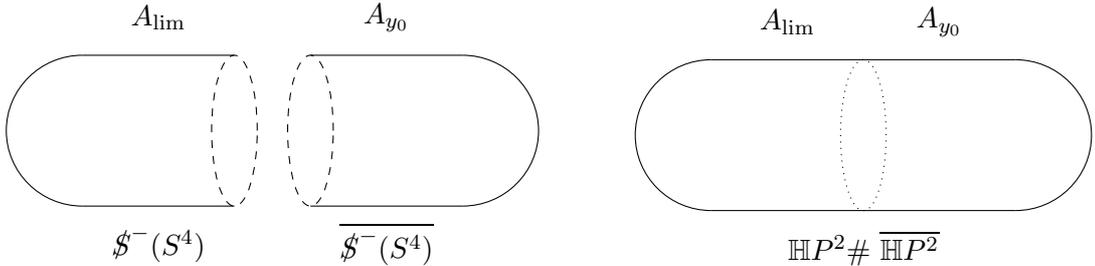\noindent
Clearly, $H^4(\mathbb{H}P^2_{(1)}) \oplus H^4(\mathbb{H}P^2_{(2)}) \cong H^4(\mathbb{H}P^2_{(1)} \#\ \overline{\mathbb{H}P^2}_{(2)})$.\par
Now, the map $\iota: S^4 \hookrightarrow \mathbb{H}P^2$ induces an isomorphism
\begin{align}\label{cohomoiso}
    \iota^* : H^4(\mathbb{H}P^2) \to H^4(S^4).
\end{align}
Then,
\begin{align}\label{manipont}
    p_1(\mathbb{H}P^2_{(1)} \#\ \overline{\mathbb{H}P^2_{(2)}}) = p_1(\mathbb{H}P^2_{(1)}) + p_1(\overline{\mathbb{H}P^2}_{(2)}) = 4a_1 + 4a_2
\end{align}
where $a_i$ is a generator of $\mathbb{H}P^2_{(i)}$ with $\displaystyle\int_{\mathbb{H}P^2_{(i)}} a_i^2 = 1$ for $i = 1,2$.\par
Let
\begin{align}\label{1stpont}
    p_1(\mathfrak{g}_{\widetilde{P}}) = n_1a_1 + n_2a_2 \in H^4(\mathbb{H}P^2_{(1)}) \oplus H^4(\mathbb{H}P^2_{(2)}).
\end{align}
Hence, from (\ref{APS7.2}), (\ref{1stpont}) and (\ref{manipont}), we have,
\begin{align}\label{limind}
    &\ \ \ind\left(\slashed{\mathfrak{D}}_{A_{\text{lim}},{CI}}^-, \slashed{S}^-(S^4), g_{CI}\right)\nonumber\\
    &= -\frac{1}{12}\int\limits_{\mathbb{H}P^2_{(1)} \#\ \overline{\mathbb{H}P^2}_{(2)}}\left(p_1(\mathfrak{g}_{\widetilde{P}})^2 - 2p_2(\mathfrak{g}_{\widetilde{P}})\right)+ \frac{1}{24}\int\limits_{\mathbb{H}P^2_{(1)} \#\ \overline{\mathbb{H}P^2}_{(2)}}p_1(\mathbb{H}P^2_{(1)} \#\ \overline{\mathbb{H}P^2}_{(2)})p_1(\mathfrak{g}_{\widetilde{P}})\nonumber\\
    &= -\frac{1}{12}\int\limits_{\mathbb{H}P^2_{(1)} \#\ \overline{\mathbb{H}P^2}_{(2)}}(n_1a_1 + n_2a_2)^2 + \frac{1}{24}\int\limits_{\mathbb{H}P^2_{(1)} \#\ \overline{\mathbb{H}P^2}_{(2)}}(4a_1 + 4a_2)(n_1a_1 + n_2a_2)\nonumber\\
    &= \frac{1}{12}\int\limits_{\mathbb{H}P^2_{(1)} \#\ \overline{\mathbb{H}P^2}_{(2)}}\left[(2n_1-n_1^2)a_1^2 + (2n_2-n_2^2)a_2^2\right]\nonumber\\
    &= \frac{1}{12}\left[(2n_1-n_1^2) + (2n_2-n_2^2)\right]
\end{align}
Now, consider a connection $\widehat{A}$ on a rank $2$ vector bundle $E$ on $\mathbb{H}P^2_{(1)} \#\ \overline{\mathbb{H}P^2}_{(2)}$ such that $E$ restricted to $\slashed{S}^-(S^4)_{(1)}$ is the instanton $A_{\text{lim}}$ on the bundle $E_1 := \pi^*(\slashed{S}^-(S^4))$ for $\pi : \slashed{S}^-(S^4)\setminus S^4 \to S^4$ and to $\slashed{S}^-(S^4)_{(2)}$ is the instanton $A_{y_0}$ on the trivial rank $2$ bundle $E_2$. Then, $\widehat{A}|_{S^4}$ for $S^4 \subset \slashed{S}^-(S^4)_{(1)}$ is the connection on $\slashed{S}^-(S^4)$ induced by the Levi-Civita connection of the anti-self-dual, Einstein metric on $S^4$ and $\widehat{A}|_{S^4}$ for $S^4 \subset \slashed{S}^-(S^4)_{(2)}$ is the flat connection on the trivial rank $2$ bundle on $S^4$ \cite{clarke-oli2020instantons}. Then, by the isomorphism (\ref{cohomoiso}), the Chern classes of $E_1$ and $E_2$ are given by $c_1(E_1) = 0, c_2(E_1) = -1$ \cite{donaldson1990geometry} and $c_1(E_2) = c_2(E_2) = 0$. This shows that
$$p_1(E) = c_2(E) = (-1)a_1 + 0 \cdot a_2 \in H^4(\mathbb{H}P^2_{(1)}) \oplus H^4(\mathbb{H}P^2_{(2)}) \cong H^4(S^4) \oplus H^4(S^4).$$
Now, clearly, $\mathfrak{g}_{\widetilde{P}} \cong \sym^2E$. Moreover, $c_2(\sym^2E) = 2c_1^2(E) + 4 c_2(E)$ \cite{Fulton1998}. Hence, 
$$c_1(\mathfrak{g}_{\widetilde{P}}) = 0, p_1(\mathfrak{g}_{\widetilde{P}}) = c_2(\mathfrak{g}_{\widetilde{P}}) = -4 \cdot a_1.$$
Hence, $n_1 = -4$ and $n_2 = 0$ in (\ref{1stpont}). Thus, from (\ref{limind}), we have,
\begin{align}\label{ind-5/2,2lim}
    \Ind_{-\frac{5}{2}} (\slashed{\mathfrak{D}}_{A_{\text{lim}}}^-, \slashed{S}^-(S^4), g) = \Ind (\slashed{\mathfrak{D}}_{A_{\text{lim}},{CI}}^-, \slashed{S}^-(S^4), g_{CI}) = -2.
\end{align}
Hence, the index of the twisted Dirac operator twisted by the limiting instanton $A_{\text{lim}}$ corresponding to the rate $-5/2$ is $-2$.
\subsection{Virtual Dimension of the Moduli Spaces}
The main result on the deformations of Clarke--Oliveira's instantons is given by the following theorem.
\begin{thm}\label{CODthm}
	The virtual dimension of the moduli space of Clarke--Oliveira's one parameter family of instantons $A_{y_0}$ with decay rate $\nu \in (-2, 0)$ is given by
	\begin{align}
		\virtualdim \mathcal{M}_{y_0}(A_\Sigma, \nu) = 
			1.
	\end{align}
    The virtual dimension of the moduli space of Clarke--Oliveira's limiting instanton $A_{\text{lim}}$ with decay rate $\nu \in (-2, 0)$ is given by
	\begin{align}
		\virtualdim \mathcal{M}_{\text{lim}}(A_\Sigma, \nu) = 
			-1.
	\end{align}
\end{thm}
\begin{proof}
     The index of the Dirac operator $\slashed{\mathfrak{D}}_{A_{y_0}}^-$ corresponding to the rate $-5/2$ is zero, which follows from (\ref{ind-5/2,2}). Moreover, from Corollary \ref{eigdircor6}, it follows that the only critical rate between $-5/2$ and $0$ is $-2$, corresponding to the eigenvalue $1/2$. Then, the result follows from the fact that the eigenspace of the eigenvalue $1/2$ is $1$-dimensional using Theorem \ref{indjump}.\par
     Similarly, since the index of the Dirac operator $\slashed{\mathfrak{D}}_{A_{\text{lim}}}^-$ corresponding to the rate $-5/2$ is $-2$, which follows from (\ref{ind-5/2,2lim}), applying the same argument, the result follows.
\end{proof}
\begin{rem}
    We note that the deformation of the one parameter family of instantons described in theorem \ref{CODthm} comes from the parameter $y_0$ in the expression of $\varphi(r)$ (\ref{valueofphi}) in Clarke--Oliveira's instanton (\ref{COinst}). Moreover, since the virtual dimension of the moduli space of the limiting instanton is negative, it is obstructed. Whereas, we still do not know whether Clarke--Oliveira's one parameter family of instantons $A_{y_0}$ is unobstructed, but we expect the instantons to be unobstructed, which will be investigated in a future project.
\end{rem}

\bibliographystyle{plain}
\nocite {*}
\bibliography{references}
\addcontentsline{toc}{section}{References}

\end{document}